\documentclass[twoside,leqno,twocolumn]{article}
\usepackage{ltexpprt}
\usepackage{hyperref}
\usepackage{url}
\usepackage[fleqn]{amsmath}
\usepackage{amsfonts}
\usepackage{algorithm}
\usepackage{algpseudocode}
\usepackage{multirow}
\usepackage{makecell}
\usepackage{amssymb}
\usepackage{graphicx}
\usepackage{diagbox}
\usepackage{verbatim}
\usepackage{mathtools}
\usepackage{xcolor}
\usepackage{etoolbox}

\def\grad{\nabla}

\def\cB{\mathcal{B}}
\def\cC{\mathcal{C}}
\def\cD{\mathcal{D}}
\def\cE{\mathcal{E}}

\def\cG{\mathcal{G}}
\def\cH{\mathcal{H}}
\def\cI{\mathcal{I}}

\def\cL{\mathcal{L}}

\def\cN{\mathcal{N}}

\def\cU{\mathcal{U}}

\def\cX{\mathcal{X}}
\def\cY{\mathcal{Y}}
\def\cZ{\mathcal{Z}}

\def\smskip{\smallskip}

\def\texitem#1{\par\smskip\noindent\hangindent 25pt
               \hbox to 25pt {\hss #1 ~}\ignorespaces}

\def\abs#1{\left|#1\right|}
\def\norm#1{\|#1\|}

\newcommand{\BEAS}{\begin{eqnarray*}}
\newcommand{\EEAS}{\end{eqnarray*}}
\newcommand{\BEA}{\begin{eqnarray}}
\newcommand{\EEA}{\end{eqnarray}}
\newcommand{\BEQ}{\begin{eqnarray}}
\newcommand{\EEQ}{\end{eqnarray}}
\newcommand{\BIT}{\begin{itemize}}
\newcommand{\EIT}{\end{itemize}}
\newcommand{\BNUM}{\begin{enumerate}}
\newcommand{\ENUM}{\end{enumerate}}

\newcommand{\BA}{\begin{array}}
\newcommand{\EA}{\end{array}}

\newcommand{\reals}{\mathbb{R}}
\newcommand{\integers}{\mathbb{Z}}

\newcommand{\Rank}{\mathop{\bf rank}}
\newcommand{\Tr}{\mathop{\bf Tr}}

\newcommand{\diag}{\mathop{\bf diag}}

\newcommand{\argmin}{\mathop{\rm argmin}}

\def\blue#1{\textcolor{blue}{#1}}

\newif\ifpagenumbering
\pagenumberingtrue

\pagenumberingfalse

\newsavebox{\assbox}
\savebox{\assbox}{\noindent\bf Assumption}
\newtheorem{assumption}{\usebox{\assbox}}

\patchcmd{\thebibliography}{\section*{\refname}}{\subsection*{\refname}}{}{}
\def\proj#1{\pi_\Omega\left(#1\right)}

\def\fprod#1{\langle#1 \rangle}
\newcommand{\admip}{\texttt{ADMIPC}}
\newcommand{\admm}{\texttt{ADMM}}
\newcommand{\rpcab}{\texttt{RPCAB}}
\newcommand{\mialm}{\texttt{M-IALM}}
\newcommand{\sgn}{\mathrm{sgn}}

\begin{document}

\title{\vspace{-0.2cm}\Large An ADMM Algorithm for Clustering Partially Observed Networks
\vspace{-0.15cm}\thanks{\vspace{-0.1cm}For all questions please contact the first author. Partially supported by NSF grant CMMI-1400217.}}
\author{N. S. Aybat\thanks{IME Dept. Penn State University. \texttt{Email:}nsa10@psu.edu.} \\
\and
S. Zarmehri\thanks{IME Dept. Penn State University. \texttt{Email:}sxz155@psu.edu.}\\
\and
S. Kumara\thanks{IME Dept. Penn State University. \texttt{Email:}skumara@psu.edu.}
}
\date{}
\maketitle

\begin{abstract} \small\baselineskip=9pt Community detection has attracted increasing attention during the past decade, and many algorithms have been proposed to find the underlying community structure in a given network. Many of these algorithms are based on modularity maximization, and these methods suffer from the resolution limit. In order to detect the underlying cluster structure, we propose a new convex formulation 
to decompose a partially observed adjacency matrix of a network into low-rank and sparse components. In such decomposition, the low-rank component encodes the cluster structure under certain assumptions. We also devise an alternating direction method of multipliers with increasing penalty sequence to solve this problem; 
 and compare it with \texttt{Louvain} method, which maximizes the modularity, on some synthetic randomly generated networks. Numerical results show that our method 
 outperforms \texttt{Louvain} method on the randomly generated networks when variance among cluster sizes increases. Moreover, empirical results also demonstrate that our formulation is indeed tighter than the robust PCA formulation, and is able to find the true clustering when the robust PCA formulation fails. 
\end{abstract}
\section{Introduction}
\label{sec:1}
Community detection or clustering is one of the most important topics in network science~\cite{Newman04_EPJ}. A cluster is defined loosely as a group of nodes which are more densely connected with each other than with nodes in the other groups of the network. Many clustering algorithms have been proposed to identify the underlying community structure in a given network. The goodness of the identified communities can be evaluated by a quality function~\cite{Santo10_Phy}. Modularity is the most popular quality function~\cite{Santo10_Phy} which was introduced by Girvan and Newman~\cite{NewGir04_phyr}. It is assumed that a higher modularity value indicates a better community structure. Although this is not always true, it has formed the motivation for developing many algorithms based on modularity maximization~\cite{Santo10_Phy}. In particular, given a partition of nodes, modularity is the sum of values, each corresponding to a group in the partition. Hence, in modularity maximization, one searches for the best partition, which is equivalent to looking for the best trade-off between the number of groups in the partition and their corresponding values; therefore, modularity maximization is an NP-complete problem~\cite{Brandes08_IEEE} and the algorithms are only able to find a good approximation to the global solution. Recently, it has been shown that modularity maximization has some problems for large networks~\cite{Santo12_phyr}. One major problem is due to \emph{resolution limit}~\cite{Santo07_pre}. Many modularity based algorithms tend to merge smaller clusters with bigger ones even when the small size cluster is a clique, and it is connected to a larger cluster by a single edge~\cite{Santo07_pre,Santo12_phyr}. This problem arises from the definition of modularity and particularly from the assumption of its null model that each node can interact with any other node in the network~\cite{Santo10_Phy}. If there are two communities with sufficiently small sizes (and hence small degrees), the expected number of edges between them for the null model is small. In this case, even the existence of a single edge between the two communities can merge them together~\cite{Santo10_Phy}. Moreover, as discussed in~\cite{Santo07_pre}, the partition corresponding to the highest modularity may not be correlated with the underlying unknown community structure. Indeed, there are some instances of real networks~\cite{Santo07_pre} and benchmark graphs~\cite{Santo12_phyr} such that the modularity maximization fails to properly identify the community structure. Recently, it has been shown that the modularity maximization problem can have different local maxima which are structurally different but have high modularity values~\cite{Good10_phy}. These solutions may disagree on many community structure properties such as the distribution of cluster sizes. This kind of disagreement may have serious impact on real world networks such as metabolic networks~\cite{Good10_phy}.

The need to find an accurate community structure motivated us to develop a new model, which depends more on the network structure and less on the quality function, together with an algorithm to solve it for community detection. Our method is based on convex optimization, and is inspired by the work in~\cite{Can09_1J}. Suppose we have a data matrix $D\in\reals^{m\times n}$ which is a summation of a \emph{low rank} matrix $\bar{L}$ and a \emph{sparse} matrix $\bar{S}$, i.e., $D=\bar{L}+\bar{S}$. Consider the following convex optimization problem:
\begin{equation}
\label{eq:rpca}
(L_{\rho}^*, S_{\rho}^*) \in\argmin_{L,S\in\reals^{m\times n}}\{ \norm{L}_{*} +  \rho \norm{S}_1:~ D = L+S \},
\end{equation}
\newpage
\noindent where $\norm{Z}_{*}:=\sum_{i=1}^{\Rank(Z)}\sigma_i(Z)$ denotes the nuclear norm of $Z\in\reals^{m\times n}$, i.e., sum of singular values of its argument, and $\norm{Z}_1=\sum_{i=1}^m\sum_{j=1}^n|Z_{ij}|$. It has been shown in~\cite{Can09_1J} that under some technical conditions on $\bar{L}$ and $\bar{S}$, problem in \eqref{eq:rpca} has a unique solution $(L^*_{\rho},S^*_{\rho})$ such that $(L^*_{\rho},S^*_{\rho})=(\bar{L},\bar{S})$ with very high probability for $\rho=1/\sqrt{\max\{m,n\}}$.

Suppose we are given an undirected network $\cG=(\cN,\cE)$, where $\cN=\{1,\ldots,n\}$ and $\cE\subset\cN\times\cN$ denote the set of nodes and edges, respectively. Suppose there are $r\ll n$ communities in $\cG$, and $\cN_\ell\subset\cN$ denotes the subset of nodes in community-$\ell$ for $1\leq\ell\leq r$. We assume that every node belongs to \emph{exactly one} community, i.e., $\bigcup_{\ell=1}^r\cN_\ell=\cN$ and $\cN_{\ell_1}\cap\cN_{\ell_2}=\emptyset$ for all $\ell_1\neq\ell_2$. Let $D\in\reals^{n\times n}$ denote the node-node incidence matrix of $\cG$ such that $D_{ii}=1$ for all $i\in\cN$, $D_{ij}=1$ if either $(i,j)\in\cE$ or $(j,i)\in\cE$, and $D_{ij}=0$ otherwise. Our idea is to decompose $D$ into a low rank matrix and a sparse matrix to recover the underlying community structure in $\cG$. Here we discuss that such decomposition is feasible under the assumption that the number of node pairs \emph{not connected} by an edge in each cluster and the number of edges connecting two different clusters are both \emph{small} for the underlying community structure in $\cG$. To motivate the upcoming discussion, first consider the scenario where the subgraph of $\cG$ restricted to $\cN_\ell$ is a clique for all $\ell=1,\cdots,r$, and there is no inter-community edge in $\cE$, i.e., $\cN_{\ell_1}\times\cN_{\ell_2}\subset \cE^c$ for all $\ell_1\neq\ell_2$ and $\cE^c$ denotes the complement of $\cE$ in $\cN\times\cN$. Clearly, $D$ is a block diagonal matrix with each block on the diagonal consisting of all ones and the off-diagonal blocks consisting of all zeros -- from now on we refer to such matrices as block diagonal matrix of ones~(BDO). Note that $D$ is a low-rank matrix such that $\Rank(D)=r$ and $\lambda_\ell(D)=|\cN_\ell|$ for all $1\leq\ell\leq r$, where $\{\lambda_\ell(D)\}_{\ell=1}^r$ denotes the non-zero eigenvalues of $D$. Hence, $D=\bar{L}+\bar{S}$ such that low-rank component $\bar{L}=D$ and sparse component $\bar{S}=\mathbf{0}_n$, where $\mathbf{0}_n\in\reals^{n\times n}$ is the matrix of zeros.

Now consider a more realistic scenario where for any $\ell\in\{1,\ldots,r\}$, a \emph{small} number of node pairs from  $\cN_{\ell}$ may not be connected by an edge, and for any $\ell_1\neq\ell_2$ there may be a \emph{small} number of edges with one end in $\cN_{\ell_1}$ and the other in $\cN_{\ell_2}$, i.e., the clusters may not be cliques, and there can be inter-cluster edges. Given $\cG$ with a non-overlapping community structure $\{N_\ell\}_{\ell=1}^r$, we define $\bar{L}=D-\bar{S}$ and $\bar{S}$ such that
\begin{equation}
\label{eq:Sbar}
\bar{S}_{ij}=\left\{
               \begin{array}{ll}
                 -1, &  \begin{array}{l}
                          \hbox{if $(i,j)\not\in\cE$, $(j,i)\not\in\cE$, and} \\
                          \hbox{$\exists \ell$ s.t. $i,j\in\cN_\ell$;}
                        \end{array}\\
                 1, & \begin{array}{l}
                          \hbox{if $(i,j)\in\cE$ or $(j,i)\in\cE$, and} \\
                          \hbox{$\exists\ell_1\neq\ell_2$ s.t. $i\in\cN_{\ell_1}$, $j\in\cN_{\ell_2}$;}
                        \end{array}\\
                 0, & \hbox{   otherwise.}
               \end{array}
             \right.
\end{equation}
 Clearly, $\bar{S}$ defined in \eqref{eq:Sbar} is \emph{sparse}, due to our assumption on the underlying community structure in $\cG$, and $\bar{L}=D-\bar{S}$ is \emph{low-rank}. Indeed, $\bar{L}$ is a BDO obtained by completing the clusters into cliques and deleting the inter-cluster edges; therefore, $\bar{L}$ is low rank because each block has rank one and $\Rank(\bar{L})$ is equal to the number of diagonal blocks, i.e., $\Rank(\bar{L})=r$. 

In this paper, we propose a convex model similar to \eqref{eq:rpca} of which optimal solution is equal to $(\bar{L},\bar{S})$, defined as in \eqref{eq:Sbar}, with very high probability. In particular, by adding constraints $L\succeq 0$, $\diag(L)=\mathbf{1}$, $L \geq 0$, and $|S_{ij}| \leq 1$ for $1\leq i \neq j \leq n$ to \eqref{eq:rpca}, we will obtain a tighter convex model. Another important property of our model is its ability to handle cases where $D$ is partially observed.
We will discuss this property in more detail in the next section, and develop an alternating direction method of multipliers~(\admm) algorithm for the proposed model.
Finally, in Section~\ref{sec:numerical}, we first discuss how to generate a random family of networks for which modularity maximization fails, then compare ours with \texttt{Louvain} method~\cite{Louvain08_stat}, which is a greedy algorithm to solve the modularity maximization problem. While we were working on this idea independently, we found out a similar work by Chen et al.~\cite{Jalali14}, which is also based on decomposition of the adjacency matrix 
by solving \eqref{eq:rpca} when $D$ is partially observed. In the next section, we will discuss the similarities among the two methods, and emphasize the advantages of our method over the one in~\cite{Jalali14}; and compare both methods in Section~\ref{sec:numerical}. Numerical results show that our model is indeed tighter than the robust PCA formulation, and is able to find the true clustering almost every time, while both \texttt{Louvain} method and the one in~\cite{Jalali14} fail when the variance among cluster sizes increases. The MATLAB code is available at \url{http://www2.ie.psu.edu/aybat/codes.html}.
\section{Methodology}
\subsection{Theoretical results}
Model~\eqref{eq:rpca} was proposed by Cand\`{e}s et al.~\cite{Can09_1J}, and it is shown under some technical conditions on the components of $D$ that its solution recovers the low rank and sparse components of the data matrix exactly with high probability - see~\cite{Can09_1J} for more details. 
When $D$ is \emph{partially} observed, 
Tao and Yuan~\cite{TaoY11_SIAM} proposed the following model:
\begin{equation}
\label{eq:rpca_p}
\min_{L,S\in\reals^{n\times n}}\{\norm{L}_{*} +  \rho \norm{S}_1:~\pi_{\Omega}(L+S)=\pi_{\Omega} (D) \},
\end{equation}
where $\Omega \subset \{(i,j): 1\leq i,j\leq n \}$ is the set of observable indices and $\pi_{\Omega}$ is the projection operator. For an arbitrary matrix $Z$, $(\pi _{\Omega}(Z))_{ij} = Z_{ij}$ when $(i,j) \in \Omega$ and $(\pi _{\Omega}(Z))_{ij} = 0$ otherwise. This model inspired us to develop a new method for network clustering based on convex optimization.

In order to compute a clustering for a partially observed network $\cG$, Chen et al.~\cite{Jalali14} proposed to solve \eqref{eq:rpca_p} repeatedly for different values of $\rho$, where $D$ is the node-node incidence matrix of $\cG$ with a diagonal of ones. In particular, the authors of~\cite{Jalali14} proposed doing bisection on $\rho$ until $L^*_\rho$ is a BDO, which they called a \emph{valid result}.
The main advantage of the proposed method in this paper over the one in~\cite{Jalali14} is that we solve a \emph{tighter} convex problem \emph{one time} with the same complexity of solving \eqref{eq:rpca}, while in~\cite{Jalali14} the authors propose to solve \eqref{eq:rpca_p} repeatedly for different values of $\rho$. Moreover, our numerical results show that our method is not only better in computation time, but also in clustering quality. Indeed, when the network size and/or the variance among cluster sizes increase, the method in~\cite{Jalali14} fails to cluster correctly while ours always succeeds.
\begin{assumption}
\label{Assum:1}
Let $\cG=(\cN,\cE)$ be an undirected network with $\cN=\{1,\ldots,n\}$, and $\cE\subset\cN\times\cN$. Suppose that $\{\cN_\ell\}_{\ell=1}^r$ be a partition of $\cN$ representing the non-overlapping communities in $\cG$. Let $\Omega:=\bar{\Omega}\cup\cD$ denote the set of observable entries of the adjacency matrix $D$ such that $\diag(D)=\mathbf{1}$, where $\cD=\{(i,i):~i\in\cN\}$, and $\bar{\Omega}\subset\cN\times\cN\setminus \cD$. Assume that an edge exists between two nodes in the same cluster with probability $p$, and it exists between two nodes from two different clusters with probability $q$ such that $p \gg q$; and for any two nodes whether there is an edge between them or not is known with probability $p_0$. 
\end{assumption}
In this paper, we propose an efficient method that can recover the underlying community structure of $\cG$ by decomposing $D$ into $(\bar{L},\bar{S})$ defined as in \eqref{eq:Sbar} with very high probability under Assumption~\ref{Assum:1}. In such decomposition, the community structure is encoded in $\bar{L}$, and the off-diagonal non-zero entries of $\bar{S}$ correspond to the node pairs that are in the same cluster but not connected by any edge, or to the edges connecting two different clusters. Following~\cite{Jalali14}, the total number of \emph{off-diagonal non-zeros} in $\bar{S}$, denoted by $\norm{\bar{S}}_{0}$, will be called total number of \emph{disagreements}. Before we introduce our model, we investigate some properties of $(\bar{L}, \bar{S})$. 
\begin{lemma}
\label{Lem:1}
Given $\cG=(\cN,\cE)$, and $\Omega\subset\cN\times\cN$, define $\chi := \{ (L,S) \in \mathbb{S}_n \times \mathbb{S}_n:~\pi_{\Omega}(L+S)=\pi_{\Omega}(D),~|S_{ij}| \leq 1 \quad \forall ~ 1\leq i \neq j\leq n,~\diag(L) = \mathbf{1},~ L \geq \mathbf{0}_n,~L\succeq \mathbf{0}_n\}$, where $\mathbb{S}_n$ denotes the subspace of $n\times n$ symmetric real matrices. Under Assumption~\ref{Assum:1}, we have $(\bar{L},\bar{S}) \in \chi$.
\end{lemma}
\begin{proof}
Directly from the definition of $\bar{S}$ in \eqref{eq:Sbar}, and from the facts: $\bar{L}:=D-\bar{S}$, $\diag(D)=\mathbf{1}$, it follows that the first four constraints are satisfied at $(\bar{L},\bar{S})$. Therefore it is enough to show that $\bar{L}$ is positive semidefinite. Under Assumption~\ref{Assum:1}, $\bar{L}$ is BDO with $r$ blocks since $\{\cN _{\ell}\}_{\ell=1}^r$ is a partition of $\cN$. Now let $v^ \ell\in\reals^n$ be such that $v_j^ \ell = 1$ if $j \in \cN _ \ell$, and $v_j^\ell = 0$ otherwise. It is easy to show that $\bar{L} = \Sigma _{\ell=1} ^{r} v^\ell (v^\ell)^T$. Therefore, $\bar{L}\succeq \mathbf{0}_n$ such that $\Rank{\bar{L}}=r$.
\end{proof}
Based on Lemma~\ref{Lem:1}, we define our new model as:
\begin{eqnarray}
\label{eq:new_tr}
\begin{aligned}
(L^*, S^*) \in \argmin_{L,S\in\mathbb{S}_n} &\Tr(L) +  \rho \norm{S}_1\\
 \hbox{s.t.\ \ } & \pi_{\Omega}(L+S)=\pi_{\Omega} (D),\\
 & \diag{(S)} = \mathbf{0},~|S_{ij}| \leq 1 \quad \forall ~ i \neq j,\\
 & L \succeq \mathbf{0}_n,~L \geq \mathbf{0}_n.
\end{aligned}
\end{eqnarray}
Note that in~\eqref{eq:new_tr}, we replaced $\norm{L}_*$ in \eqref{eq:rpca_p} with $\Tr(L)$, and also replaced $\diag(L) = \mathbf{1}$ constraint in the definition of $\chi$ with $\diag(S) = \mathbf{0}$ constraint. The first follows from the fact that $L\succeq\mathbf{0}_n$ implies $\norm{L}_*=\Tr(L)$. Indeed, for a positive semidefinte matrix $L$, its non-zero singular values $\{\sigma_i\}_{i=1}^{\Rank(L)}$ are equal to its non-zero eigenvalues $\{\lambda_i\}_{i=1}^{\Rank(L)}$. Therefore, we have $\norm{L}_* = \sum _{i=1}^{n}\sigma_i = \sum _{i=1}^{n}\lambda_i=\Tr(L)$.
Moreover, since $\cD\subset\Omega$ and $\diag(D)=\mathbf{1}$, $\diag(L)=\mathbf{1}$ if and only if $\diag(S) = \mathbf{0}$.
Note that replacing $\diag(L) = \mathbf{1}$ equivalently with $\diag(S) = \mathbf{0}$ is the key point for 
developing an \admm~algorithm with efficiently solvable subproblems, which will be discussed in the next section.



The following two theorems show the importance and special properties of our formulation. Theorem~\ref{Thm:1} and Theorem~\ref{Thm:2} were originally proved in~\cite{Jalali14} for model~\eqref{eq:rpca_p}. Since the feasible region of our model in~\eqref{eq:new_tr} is a subset of that in~\eqref{eq:rpca_p}, these important results trivially extend to our formulation as well.
\begin{theorem}
\label{Thm:1}
For any $\rho > 0$, if $L^*$ in \eqref{eq:new_tr} is a BDO, then it provides the optimal clustering in the sense that the total number of observed disagreements, i.e., $\norm{\pi_{\Omega}(S^*)}_0$, is minimized.
\end{theorem}
\begin{proof}
See proof of Theorem~2 in~\cite{Jalali14}.
\end{proof}

\begin{theorem}
\label{Thm:2}
Given $\cG=(\cN,\cE)$ and $\Omega\subset\cN\times\cN$ satisfying Assumption~\ref{Assum:1}. 
Let $\{\cN_{\ell} \} _{\ell=1}^r$ represent the true underlying community structure of $\cG$, and $\bar{L}=D-\bar{S}$, where $\bar{S}$ is defined in \eqref{eq:Sbar}. Then for all $c>0$, there exists $C>0$ such that with probability of at least $1-cn^{-10}$, $(\bar{L}, \bar{S})$ is the unique optimal solution of \eqref{eq:new_tr} when $\rho = \frac{1}{32 \sqrt{np_0}}$, provided that
\begin{eqnarray}
\label{eq:n_cond}
 n \log ^2 n\leq C K_{\min} ^2~p_0 (1-2 \gamma)^2,
\end{eqnarray}
where $\gamma = \max \{1-p,q \}$ and $K_{\min}:=\min\{|\cN_\ell|:~1\leq\ell\leq r\}$ is the size of the smallest cluster.
\end{theorem}
\begin{proof}
Given $c>0$, Theorem 2 in~\cite{Jalali14} shows that there exists $C>0$ such that $(\bar{L}, \bar{S})$ is the unique optimal solution to \eqref{eq:rpca_p} with probability at least $1-cn^{-10}$ provided that \eqref{eq:n_cond} holds. Lemma~\ref{Lem:1} implies that $(\bar{L}, \bar{S})$ is feasible to \eqref{eq:new_tr}; hence, it must be an optimal solution to the more tighter problem in \eqref{eq:new_tr} as well. Moreover, under the assumptions of Theorem~\ref{Thm:2}, \eqref{eq:rpca_p} has a unique solution with high probability, which implies that $(\bar{L}, \bar{S})$ is the unique optimal solution to \eqref{eq:new_tr} w.p. at least $1-cn^{-10}$.
\end{proof}
\subsection{Algorithm}
In this section, we develop an \admm~algorithm to solve \eqref{eq:new_tr}. Define $\phi\subset\mathbb{S}_n\times\mathbb{S}_n$ as
\begin{equation}
\label{eq:phi}
\phi := \left\{(X,S):
\begin{array}{l}
  \proj{X+S}=\proj{D}, \\
  X \geq \mathbf{0}_n,\ \diag{(S)} = \mathbf{0},\\
  |S_{ij}| \leq 1, \quad 1\leq i \neq j\leq n
\end{array}
\right\}.
\end{equation}
By using partial variable splitting as in~\cite{Aybat14_arxiv,Aybat14_COAP}, \eqref{eq:new_tr} can be written equivalently as follows:
\begin{eqnarray}
\label{eq:parS}
\begin{aligned}
(L^*, L^*, S^*) \in \argmin_{L,X,S \in \mathbb{S}_n} & \Tr(L) +  \rho \norm{S}_1\\ 
\hbox{s.t.\ \ } & ~X = L,~L \succeq \mathbf{0}_n,~(X,S)\in\phi.
\end{aligned}
\end{eqnarray}
Let $\mathbb{S}_n^{+}$ denote the cone of $n\times n$ symmetric positive semidefine matrices. Given a penalty parameter $\mu>0$, the partial augmented Lagrangian~\cite{Bert82} of \eqref{eq:parS} is defined for any $L\in\mathbb{S}_n^{+}$, $(X,S)\in\phi$, and $Y\in\mathbb{S}_n$ as follows
\begin{eqnarray*}
\lefteqn{\cL_{\mu}(L,X,S;Y) =}\\
& & \Tr(L)+\rho \norm{S}_{1}+\fprod{Y,X-L}+ \tfrac{\mu}{2}\norm{X-L}_{F}^2.
\end{eqnarray*}
Given $Y\in\mathbb{S}_n$, since it is not easy to minimize $\cL_{\mu}(L,X,S;Y)$ jointly in $(L,X,S)\in\mathbb{S}_n^{+}\times\phi$, the \emph{method of multipliers} is not a practical approach to solve \eqref{eq:parS}. On the other hand, given $Y$, alternating minimization of $\cL_{\mu}(L,X,S;Y)$ in $(X,S)\in\phi$ for fixed $L$, and in $L\in\mathbb{S}_n^{+}$ for fixed $(X,S)$ can be done efficiently. Therefore, we propose \admip, which is an \admm~algorithm with increasing penalty sequence, to solve \eqref{eq:parS}. Each step of \admip~is displayed in Figure~\ref{al:ADMM}. The subproblems in Step~\ref{algeq:L_problem} and Step~\ref{algeq:XS_problem} are the computational bottlenecks, and they can be solved efficiently as explained in Lemma~\ref{Lem:2} and Lemma~\ref{Lem:3}. The initialization part will be discussed in Section~\ref{sec:initialization}.

The convergence of \admip~directly follows from \cite{He02_1J}. Indeed, the
     variable penalty \admm~algorithms in~\cite{He98_1J,He00_1J,He02_1J}
     are proposed to solve variational inequalities~(VI) of the form:
  \begin{align*}
  &(x-x^*)^\top F(x^*)+(y-y^*)^\top G(y^*) \geq 0,\ \forall
  (x,y)\in\Omega\\
  &\Omega:=\{(x,y):~x\in\cX,~y\in\cY,~Ax+By=b\}, 
  \end{align*}
  where $A\in\reals^{m\times n_1}$, $B\in\reals^{m\times n_2}$, and
  $b\in\reals^m$. The convergence proofs in~\cite{He98_1J,He00_1J,He02_1J}
  require that both $F:\cX\rightarrow\reals^{n_1}$ and
  $G:\cY\rightarrow\reals^{n_2}$ be \emph{continuous point-to-point maps}
  that are monotone with respect to the non-empty closed convex sets
  $\cX\subset\reals^{n_1}$ and $\cY\subset\reals^{n_2}$, respectively. 
  When these variable penalty \admm~methods for VI are applied to the VI
  reformulation of
  convex optimization problems of the form $\min\{f(x)+g(y):\
  (x,y)\in\Omega\}$, the requirement that $F$ and $G$ be continuous
  point-to-point maps
  implies that $F(x)=\grad f(x)$, and $G(y)=\grad g(y)$. 
  On the other hand, when $f$ (similarly $g$) is a non-smooth
  convex function, 
  $F$ (similarly $G$) is the subdifferential operator, which
  is a \emph{point-to-set map}; 
  therefore, the convergence proofs for variable penalty \admm~ algorithms
  in~\cite{He98_1J,He00_1J,He02_1J} do not 
  extend to non-smooth convex optimization problems
  -- see Assumption~A and the following discussion on page 107 in
  \cite{He02_1J}. However, even though the objective in \eqref{eq:parS} is non-smooth, the following result establishes that the convergence of \admip~follows from \cite{He02_1J}.
\begin{theorem}
\label{thm:comvergence}
Let $Z_k=(L_k,X_k,S_k,Y_k)$ denote the iterates generated by \admip~in Figure~\ref{al:ADMM}, and $\cZ^*$ denote the set of optimal primal-dual pairs to \eqref{eq:parS}, i.e., $(L^*,X^*,S^*,Y^*)\in \cZ^*$ if and only if
\begin{equation*}
\begin{array}{l}
\fprod{\mathbf{I}_n-Y^*, L-L^*}\geq 0, \quad\forall~L\succeq\mathbf{0}_n,\\
\rho\fprod{G, S-S^*}+\fprod{Y^*, X-X^*}\geq 0,\quad \forall~(X,S)\in\phi,\\
X^*=L^*,\quad G\in\partial\norm{S^*}_1.
\end{array}
\end{equation*}
Then $\min\{\norm{Z_k-Z}_F:~Z\in\cZ^*\}\rightarrow 0$. Moreover, $\{Z_k\}$ is bounded.
\end{theorem}
\begin{proof}
Using the change of variables $S:=S^+-S^-$ for $S^+,S^-\geq\mathbf{0}_n$, $\norm{S}_1$ can be equivalently written as $\fprod{\mathbf{E}_n, S^++S^-}$, where $\mathbf{E}_n\in\reals^{n \times n}$ is a matrix of ones. Consider
\begin{equation}
\label{eq:smooth_equivalent}
\min\left\{\Tr(L) +  \rho \fprod{\mathbf{E}_n, S^++S^-}:
\begin{array}{l}
X = L,~L \succeq \mathbf{0}_n,\\
(X,S^+,S^-)\in\phi'
\end{array}\right\},
\end{equation}
where $\phi'\subset\prod_{i=1}^3\mathbb{S}_n$ is defined as
  $$\phi':=\left\{(X,S^+,S^-):\hspace{-1mm}
\begin{array}{l}
  \proj{X+S^+-S^-}=\proj{D}, \\
  X\geq\mathbf{0}_n,~S^+\geq\mathbf{0}_n,~S^-\geq\mathbf{0}_n,\\
  \diag(S^+)=\diag(S^-)=\mathbf{0},\\
  S^+_{ij}+S_{ij}^-\leq 1,\ \forall i\neq j
\end{array}\right\}.$$
Note that \eqref{eq:smooth_equivalent} is a smooth convex optimization problem equivalent to \eqref{eq:parS}, and it satisfies all the assumption in~\cite{He02_1J}. Given a nondecreasing penalty sequence $\{\mu_k\}$ such that $\sup_k \mu_k<\infty$, let $\{(\tilde{L}_k,\tilde{X}_k,\tilde{S}^+_k,\tilde{S}_k^-,\tilde{Y}_k)\}$ be the iterate sequence generated by the variable penalty \admm~in~\cite{He02_1J} when the augmented Lagrangian of \eqref{eq:smooth_equivalent} is minimized alternatingly in $(X,S^+,S^-)\in\phi'$, and in $L \succeq \mathbf{0}_n$. Define $\tilde{Z}_k:=(\tilde{L}_k,\tilde{X}_k,\tilde{S}_k,\tilde{Y}_k)$, where $\tilde{S}_k:=\tilde{S}_k^+-\tilde{S}_k^-$. It is easy to see that $\{\tilde{Z}_k\}$ would be the same with the one generated by \admip~in Figure~\ref{al:ADMM}, i.e., $\tilde{Z}_k=Z_k$ for all $k\geq 1$. The result follows from Theorem~4 in~\cite{He02_1J}, which shows that $\{\tilde{Z}_k\}$ is bounded and $\min\{\norm{\tilde{Z}_k-Z}_F:~Z\in\cZ^*\}\rightarrow 0$.
\end{proof}
\begin{figure}[h!]
    \rule[0in]{3.25in}{1pt}\\
    \textbf{Algorithm \admip~$\left(\rho,~\{\mu_k\}_{k\in\integers_+}\right)$}\\
    \rule[0.125in]{3.25in}{0.1mm}
    \vspace{-0.1in}
{\footnotesize
\begin{algorithmic}[1]
\State \textbf{Input}: $\rho>0$, $\{\mu_k\}_{k\in\integers_+}\subset\reals_{++}$ s.t.
$\mu_{k+1}\geq\mu_k$, and $\sup_k\mu_k<\infty$
\State \textbf{Initialization:} $k=0;~L_0 = \mathbf{0}_n;$\\ $Y_0 = \frac{\pi_{\Omega}(D)}{\max\left\{ \norm{\pi_{\Omega}(D)}_2,~\rho^{-1}\norm{\pi_{\Omega}(D)}_{\infty} \right\}};$
   \While{not converged}
      \State $(X_{k+1},S_{k+1})\gets \argmin\{\rho \norm{S}_{1}+ \frac{\mu_k}{2}\norm{X- L_{k}+\frac{Y_k}{\mu_k}}_{F}^2:\ (X,S) \in \phi \}$ \label{algeq:XS_problem}
      \State $L_{k+1}\gets \argmin\{\Tr(L)+ \frac{\mu_k}{2}\norm{L-X_{k+1}-\frac{Y_k}{\mu_k}}_{F}^2:~L\succeq 0\}$ \label{algeq:L_problem}
      \State $Y_{k+1}\gets Y_k+\mu_k(X_{k+1}-L_{k+1})$
      \State $k\gets k+1$
   \EndWhile\label{euclidendwhile}
\end{algorithmic}
}
\rule[0.125in]{3.25in}{0.1mm}
    \vspace*{-0.25in}
    \caption{\admip: Alternating Direction Method with Increasing Penalty for Clustering}\label{al:ADMM}
\end{figure}

Lemma~\ref{Lem:2} shows that the subproblem in Step~\ref{algeq:L_problem} can be solved efficiently by computing a partial-eigenvalue decomposition of an $n\times n$ matrix.
\begin{lemma}
\label{Lem:2}
The solution to the subproblem in Step~\ref{algeq:L_problem} can be written in closed form:
\begin{eqnarray}
\label{eq:Lstar}
\qquad L_{k+1} = W \diag\left(\max\left\{\lambda_k-\mu_k^{-1}\mathbf{1},\mathbf{0}\right\}\right)W^T,
\end{eqnarray}
where $W \diag(\lambda_k)W^T$ is the eigenvalue decomposition of $Q_k^X:=X_{k+1}+\dfrac{Y_k}{\mu_k}$.
\end{lemma}
\begin{proof}
Since $\Tr(L)=\fprod{\mathbf{I}_n, L}$, the subproblem in Step~\ref{algeq:L_problem} can be equivalently written as
\begin{equation}
\label{eq:proofL}
L_{k+1} 
 = \argmin_{L\succeq 0}\norm{L-\left(Q_k^X-\mu_k^{-1}\mathbf{I}_n\right)}_F.
\end{equation}
Since $Q_k^X-\mu_k^{-1}\mathbf{I}_n=W \diag(\lambda_k-\mu_k^{-1}\mathbf{1})W^T$, \eqref{eq:Lstar} follows from the properties of Euclidean projection onto the positive semidefinite cone of symmetric matrices.
\end{proof}
One of the main reasons of using an increasing sequence of penalties $\{\mu_k\}$ in \admip~is because the work required for eigenvalue decomposition in Step~\ref{algeq:L_problem} reduces significantly as fewer leading eigenvalues are needed for small values of $\mu_k$. Note that according to \eqref{eq:Lstar}, we do not need to compute eigenvalues of $Q_k^X$ that are smaller than $\mu_k^{-1}$. Indeed, $Q_k^X$ may not be low rank, and many of its eigenvalues may be large during the initial iterations in the transient phase of the algorithm. This could make Step~\ref{algeq:L_problem} an expensive operation for a constant penalty \admm~method with $\mu_k=\mu$, as there may be many leading eigenvalues that are larger than $\mu$. However, based on \eqref{eq:Lstar}, by choosing small values for $\mu_k$ during the initial iterations and then gradually increasing it, we can avoid computing all the eigenvalues of $Q_k^X$. Refer to \cite{Aybat14_arxiv} for more details about this concept. Moreover, it is shown in \cite{He02_1J} that results of Theorem~\ref{thm:comvergence} are still true if Step~\ref{algeq:L_problem} is computed inexactly. Since with very high probability the optimal solution $L^*$ is unique and equal to $\bar{L}$, which has low rank, Step~\ref{algeq:L_problem} can be computed approximately by calculating only a small number of leading eigenvalues, and the approximation error will be small for all sufficiently large $k$. Indeed, Theorem~\ref{thm:comvergence} implies that with high probability $X_k\rightarrow L^*=\bar{L}$ (due to uniqueness of $L^*$). Hence, for any $\delta>0$, there exists $\{\mu_k\}$ such that $\norm{Q_k^X-X_{k+1}}_2\leq \tfrac{\delta}{2}$ due to boundedness of $\{Y_k\}$, and $\norm{X_{k+1}-\bar{L}}_2\leq\tfrac{\delta}{2}$, where $\norm{.}_2$ denotes the spectral norm. Thus, $\norm{Q^X_k-\bar{L}}_2\leq\delta$. Moreover, since eigenvalues of a matrix is a continuous function of its entries, it follows that for all $\delta>0$, there exists $\{\mu_k\}$ such that $\norm{\lambda_k-\bar{\lambda}}_\infty\leq\delta$, where $\bar{\lambda}\in\reals^n$ denotes the vector of eigenvalues of $\bar{L}$. Since the number of nonzeros in $\bar{\lambda}$ is $r\ll n$, $n-r$ components of $\lambda_k$ is between $-\delta$ and $\delta$.

In order to compute the eigenvalue decomposition of $Q_k^X$ in Step~\ref{algeq:L_problem}, we used LANSVD routine in PROPACK package. LANSVD routine is based on the Lanczos bidiagonalization algorithm with partial reorthogonalization for computing partial singular value decomposition~(SVD).
Let $\tilde{\lambda}:=\lambda_k-\mu_k^{-1}\mathbf{1}$, and 
$U\diag(\sigma)V^T$ denote SVD of $Q_k^X-\mu_k^{-1}\mathbf{I}_n$, where $U_i$, $V_i$ denote left and right singular vectors corresponding to $i$-th singular value $\sigma_i$. It is clear that if $U_i^T V_i=1$, then $\tilde{\lambda}_i=\sigma_i>0$; and if $U_i^TV_i=-1$, then $\tilde{\lambda}_i=-\sigma_i<0$. Hence, \eqref{eq:Lstar} can be computed efficiently using a partial SVD.
\begin{lemma}
\label{Lem:3}
Let $\Omega:=\bar{\Omega}\cup\cD$ denote the set of observable entries of the adjacency matrix $D$ such that $\diag(D)=\mathbf{1}$, where $\cD=\{(i,i):~i\in\cN\}$, and $\bar{\Omega}\subset\cN\times\cN\setminus \cD$. The solution to the subproblem in Step~\ref{algeq:XS_problem} can be written in closed form:
\begin{eqnarray*}
\begin{aligned}
C_1 &= \sgn\left(\pi_{\bar{\Omega}}(D-Q_k^L)\right),\\
C_2 &= \max \{ \abs{\pi_{\bar{\Omega}}(D-Q_k^L)}-\rho\mu_k^{-1}\mathbf{E}_n, \mathbf{0}_n\}, \\
S_{k+1} & = \min \{\pi_{\bar{\Omega}}(D),~\max \{-\mathbf{E}_n,~C_1\odot C_2\}\},\\
X_{k+1} & = \pi_{\Omega}(D-S_{k+1})+ \max \{ \pi_{\Omega^c}(Q_k^L),\mathbf{0}_n \},
\end{aligned}
\end{eqnarray*}
where $Q_k^L:=L_k-\frac{Y_k}{\mu_k}$, $\mathbf{E}_n\in\reals^{n \times n}$ is a matrix of ones, and $\odot$ represents the component-wise multiplication. 
\end{lemma}
\begin{proof}
The subproblem in step 5 can be written as
\begin{equation}
\label{eq:proofS_1}
(X_{k+1},S_{k+1}) = \argmin_{(X,S) \in \phi } \rho \norm{S}_{1}+\frac{\mu_k}{2}\norm{X-Q_k^L}_{F}^2.
\end{equation}
For $(X,S) \in \phi$, we have
\begin{eqnarray}
\label{eq:proofS_2}
\begin{aligned}
\lefteqn{X-Q_k^L  = \pi_{\bar{\Omega}}(D-S-Q_k^L)}\\ 
& &\mbox{}+\pi_{\cD}(D-Q_k^L)+\pi_{\Omega^c}(X-Q_k^L). 
\end{aligned}
\end{eqnarray}
Moreover, from the optimality conditions for \eqref{eq:proofS_1}, it is clear that $(S_{k+1})_{ij}=0$ for all $(i,j)\in\cD\cup\Omega^c$. Therefore, 
\eqref{eq:proofS_1} is equivalent to the following problem:
%
%
\begin{eqnarray}
\label{eq:proofS_5}
\begin{aligned}
\min_{(X,S) \in \mathbb{S}_n\times\mathbb{S}_n}~ & \rho \norm{\pi_{\bar{\Omega}}(S)}_{1} + h(X,S)\\
\hbox{s.t.\ \ } & \hspace{0.2cm} \ \ \qquad \mathbf{0}_n \leq \pi_{\Omega^c}(X),\\ 
&-\pi_{\bar{\Omega}}(\mathbf{E}_n) \leq \pi_{\bar{\Omega}}(S) \leq \pi_{\bar{\Omega}}(D), 
\end{aligned}
\end{eqnarray}
where $h(X,S):=\tfrac{\mu_k}{2}\norm{\pi_{\bar{\Omega}}(S)-\pi_{\bar{\Omega}}(D-Q_k^L)-\pi_{\Omega^c}(X-Q_k^L)}_{F}^2$.
For the sake of notational simplicity, let $\tilde{S}_{ij} := (D-Q_k^L)_{ij}$ for all $(i,j)\in\bar{\Omega}$. Note that \eqref{eq:proofS_5} is separable over $(i,j)$. Therefore, for all $(i,j)\in\bar{\Omega}$, 
\begin{eqnarray}
\label{eq:proofS_6}
\begin{aligned}
\quad (S_{k+1})_{ij} = \argmin_{S_{ij} \in \reals} ~ & \rho \abs{S_{ij}}+ \frac{\mu _{k}}{2}(S_{ij}-\tilde{S}_{ij})^2\\
\hbox{\ s.t. \ \ } & -1 \leq S_{ij} \leq D_{ij}.
\end{aligned}
\end{eqnarray}
Given $\bar{t}\in\reals$ and $\mu>0$, define $f:\reals \rightarrow \reals$ such that $f(t)=\rho \abs{t}+\frac{\mu}{2}(t-\bar{t})^2$. 
 Let $t_{u}^* := \argmin_{t \in \reals} f(t)$ and $t_{c}^* = \argmin_{t \in \reals}\{ f(t):~ a \leq t \leq b\}$. Since $f$ is convex on $\reals$, it is easy to show that
\begin{eqnarray}
\label{eq:proofS_7}
\begin{aligned}
t_{u}^* &= \sgn(\bar{t})~\max\left\{ \abs{\bar{t}}-\tfrac{\rho}{\mu},0\right\},\\
t_c^* &= \min\{b, \max\{a,t_u^*\}\}.
\end{aligned}
\end{eqnarray}
Thus, \eqref{eq:proofS_7} implies that for all $(i,j) \in \bar{\Omega}$, we have
\begin{eqnarray}
\label{eq:proofS_9}
\begin{aligned}
\quad (S_{k+1})_{ij} &= \min\left\{D_{ij},~\max\left\{-1,~c_{ij} \right\}\right\},\\
c_{ij} &= \sgn(\tilde{S}_{ij})\max\left\{\abs{\tilde{S}_{ij}}-\tfrac{\rho}{\mu_k},~0\right\}.
\end{aligned}
\end{eqnarray}
The structure of $S_{k+1}$ follows from \eqref{eq:proofS_9} and the fact that $(S_{k+1})_{ij}=0$ for all $(i,j)\in\cD\cup\Omega^c$. 

Since $(X_{k+1},S_{k+1})\in\phi$, clearly for all $(i,j)\in\Omega$, we have $(X_{k+1})_{ij}=(D-S_{k+1})_{ij}$. Moreover, \eqref{eq:proofS_2} implies that for all $(i,j)\in\Omega^c$, we have
\begin{eqnarray}
\label{eq:proofS_11}
\begin{aligned}
(X_{k+1})_{ij}&=\argmin_{X_{ij}\geq 0}\left(X_{ij}-(Q_k^L)_{ij}\right)^2,\\
&=\max \{ (Q_k^L)_{ij}, 0 \}.
\end{aligned}
\end{eqnarray}
\end{proof}
\section{Numerical results}
\label{sec:numerical}
In Section~\ref{sec:RPCA}, we compared our formulation~\eqref{eq:new_tr} with the robust PCA formulation~\eqref{eq:rpca_p}, which is adopted by Chen et al. in~\cite{Jalali14}. Numerical results show that our formulation is more tighter, and is able to recover many clusters which cannot be detected using the methodology given in~\cite{Jalali14}. Next, in Section~\ref{sec:modularity_max}, we compared \admip~with \texttt{Louvain} method, which is based on modularity maximization, on randomly generated test problems. The results show that as the number of nodes in the network increases, \texttt{Louvain} method starts merging small clusters; and this phenomena becomes more apparent when the variation among cluster sizes increases. The empirical results presented in Section~\ref{sec:modularity_max} indeed confirm that resolution limit~\cite{Santo07_pre,Santo12_phyr} becomes a major drawback for modularity maximization.
\subsection{Random network generation.}
\label{sec:generation}
In this section we describe the random network generation used in our experiments. Let $\cG=(\cN,\cE)$ be a random undirected network, and $\{\cN_\ell\}_{\ell=1}^r$, a partition of $\cN$, be the underlying clustering in $\cG$ chosen such that 
\begin{equation*}
\cN_{\ell}:=\left\{\sum_{i=1}^{\ell-1}n_i+1,\ldots,\sum_{i=1}^{\ell}n_i\right\}, \quad \forall \ell\in\{1,\ldots,r\},
\end{equation*}
where $n_\ell:=|\cN_\ell|$ denote the size of $\ell$-th cluster. Let $0<\alpha\leq 1$ be a parameter that will control the variation among cluster sizes $\{n_\ell\}_{\ell=1}^r$. Note that $n = \sum_{\ell=1}^r \frac{1-\alpha}{1-\alpha^r}~n\alpha ^{\ell-1}$. Given $x>0$, let $[x]$ denote the nearest integer to $x$. The cluster sizes are chosen as
\begin{equation}
\label{eq:size}
n_\ell = \left[\frac{1-\alpha}{1-\alpha^r}~n\alpha ^{\ell-1}\right],\quad \forall \ell\in\{1,\ldots,r\}.
\end{equation}
In our experiments, we choose $r=\lceil0.05 n\rceil$. For instance, suppose $n=100$, then the total number of clusters is $r= 5$; Table~\ref{$alpha$-table} displays the size of each cluster for different values of $\alpha$. For the sake of simplicity, assume that $n=\sum_{\ell=1}^r n_\ell$, and $n_\ell>0$ for all $\ell=1,\ldots,r$; otherwise, we define $\cN_\ell$ only for $\ell$ such that $n_\ell>0$, and reset $r=\big|\{\ell:\ n_\ell>0\}\big|$.
\begin{table}[h!]
\label{$alpha$-table}
\begin{center}
{\scriptsize
\begin{tabular}{c rrrrr}
$\alpha$ & $n_1$ & $n_2$& $n_3$& $n_4$& $n_5$\\ \hline
\textbf{1}    &20 &20 &20 &20 &20 \\
\textbf{0.9}  &24 &22 &20 &18 &16\\
\textbf{0.8}  &30 &24 &19 &15 &12\\
\textbf{0.7} &36 &25 &18 &12 &9\\
\textbf{0.6} &43 &26 &16 &9  &6\\
\textbf{0.5} &52 &26 &13 &6  &3\\
\end{tabular}
}
\end{center}
\vspace{-0.5cm}
\caption{Cluster sizes for different values of $\alpha$ when $n=100$ and $r=5$.}
\end{table}

In the next step, after we choose the underlying clustering $\{\cN_\ell\}_{\ell=1}^r$ as above, we generated the edges in $\cG$ as follows. Let $\cU=\{(i,j)\in\cN\times\cN:\ i<j\}$, and $\cE^U\subset\cU$ be such that $|\cE^U|=\big[0.05 |\cU|\big]$ elements are randomly chosen with equal probability; define $\cE^L:=\{(i,j): (j,i)\in\cE^L\}$, and
$\bar{\cE}:=\{(i,j)\in\cN\times\cN:\ \exists\ell\in\{1,\ldots,r\} \hbox{ s.t. } i\in\cN_\ell,~j\in\cN_\ell\}$. Then we set $\cE$ as the symmetric difference of $\bar{\cE}$ and $\cE^U\cup\cE^L$, i.e., $\cE:=\bar{\cE}\Delta(\cE^U\cup\cE^L)$. Note that $\cG=(\cN,\cE)$ generated this way satisfies Assumption~\ref{Assum:1}.

Let $D\in\mathbb{S}_n$ be the node-node incidence matrix corresponding to $\cG$ such that $\diag(D)=\mathbf{1}$, i.e., $D_{ij}=1$ if $(i,j)\in\cE$ or $i=j$, and $D_{ij}=0$ otherwise. Let $\Omega\subset\cN\times\cN$ be the set of indices corresponding to the observable entries of $D$. Note that according to Assumption~\ref{Assum:1}, for any given $i\in\cN$ and $j\in\cN$, whether $i$ and $j$ are connected by an edge in $\cE$ or not is known with probability $p_0$. Hence, to generate $\Omega$, let $\Omega^U$ be the set of $N_o:=\left[\frac{p_0(n^2-n)}{2}\right]$ indices of the upper triangular entries of $D$ chosen uniformly at random, i.e., $\Omega^U\subset\cU$ such that $|\Omega^U|=N_o$. Since $D$ is symmetric, the symmetric lower triangular elements $\Omega^L:=\{(i,j):\ (j,i)\in\Omega^U\}$ should be in $\Omega$; and it's also known that the diagonal entries are all ones. Therefore, we set $\Omega:=\Omega^U\cup\Omega^L\cup\cD$.
\subsection{Initialization and stopping criterion.}
\label{sec:initialization}
In all the experiments, we set $\rho=\frac{1}{\sqrt{n}}$, $L_0 = \mathbf{0}_n$, $\mu_0 = \frac{1.25}{\norm{\pi_{\Omega}(D)}_2}$, and $Y_0 = \frac{\pi_{\Omega}(D)}{\max\left\{ \norm{\pi_{\Omega}(D)}_2,~\rho^{-1}\norm{\pi_{\Omega}(D)}_{\infty} \right\}}$ in \admip. 
The penalty multiplier sequence $\{ \mu_{k} \}$ is chosen such that $\mu_{k+1} = \min \{ \kappa \mu _k, \bar{\mu}\}$ for $k \geq 1$, where $\bar{\mu} = 10^7$ and $\kappa = 1.2$. We terminate the algorithm when the following primal-dual stopping conditions hold: 
\begin{equation}
\label{eq:tol}
\norm{L_{k+1}-X_{k+1}}_F \leq \mathbf{tol}_p,
~\frac{\mu_k \norm{L_{k+1}-L_{k}}_F}{\norm{\proj{D}}_F} \leq \mathbf{tol}_d,
\end{equation}
where $\mathbf{tol}_p = \epsilon _r \max\{\norm{L_{k+1}}_F,\norm{X_{k+1}}_F\}$, $\mathbf{tol}_d = \epsilon _r \norm{Y_{k+1}}_F$, and $\epsilon _r=5\times 10^{-4}$. The first equation gives the primal stopping criterion and the second one gives the dual stopping criterion. For more details about the stopping criteria refer to \cite{Boyd11}.
\subsection{Results}
All the numerical experiments were conducted on a Windows 7 machine with Intel Core i7-3520M Processor (4 MB cash, 2 cores at 2.9 GHz), and 16 GB RAM running MATLAB 8.2 (64 bit).
We consider two different cases. In the first case, we assume that $\cE$ is perfectly known, i.e., all the entries of $D$ are observed. In the second case, we assume that $\cE$ is partially observable, i.e., we only know the entries of $D$ corresponding to indices in $\Omega$. For both cases, we compared \admip~with the method proposed in~\cite{Jalali14} and with \texttt{Louvain} method for different values of $(n,\alpha)$.
\subsubsection{ADMIPC vs ADMM on RPCA}
\label{sec:RPCA}
First, for $\rho = \frac{1}{\sqrt{n}}$, we compare our formulation, given in \eqref{eq:new_tr}, with the robust PCA~(RPCA) formulation in~\eqref{eq:rpca_p} to check whether the proposed formulation \eqref{eq:new_tr} is tighter than \eqref{eq:rpca_p}. In particular, given randomly generated networks as described in Section~\ref{sec:generation}, we solve \eqref{eq:new_tr} using \admip~and compare the results with those obtained by solving \eqref{eq:rpca_p} 
using a modified version of IALM in~\cite{Lin2013_arxiv}. IALM is nothing but an increasing penalty \admm~method customized for \eqref{eq:rpca}. Note that IALM~\cite{Lin2013_arxiv} works when $D$ is fully observed, and it does not work on \eqref{eq:rpca_p}. However, Theorem~1.1 in~\cite{Aybat14_COAP} shows that \eqref{eq:rpca_p} is equivalent to
\begin{equation}
\label{eq:rpca_p_equiv}
\min_{L,S\in\reals^{n\times n}} \{\norm{L}_{*} + \rho \norm{\pi_{\Omega}(S)}_1:~L+S=\pi_{\Omega} (D) \};
\end{equation}
and one can easily modify IALM~\cite{Lin2013_arxiv} to solve \eqref{eq:rpca_p_equiv}. We call the modified version as \mialm$(\rho)$. The results presented in this section show that, for $\rho=1/\sqrt{n}$, our formulation \eqref{eq:new_tr} is indeed tighter than the RPCA formulation \eqref{eq:rpca_p}.

Next, we compared \admip~with the method developed in~\cite{Jalali14}, which is based on RPCA formulation \eqref{eq:rpca_p}. We call the method in~\cite{Jalali14} as RPCA with bisection~(\rpcab). \rpcab~calls \mialm~on \eqref{eq:rpca_p} for changing values of $\rho$. In particular, for a given $\rho>0$, \rpcab~calls \mialm~to compute $L^*_\rho$, the optimal low-rank component to \eqref{eq:rpca_p}. Next, if $\Tr(L^*_\rho)\neq n$, then \rpcab~updates $\rho$ as follows: when $\Tr(L^*_\rho)>n$, $\rho\gets\rho/2$; otherwise, $\rho\gets2\rho$. After $\rho$ is updated, \rpcab~calls \mialm~on \eqref{eq:rpca_p} with the new $\rho$ value. 
In all the numerical tests we set the initial value of $\rho = \frac{1}{\sqrt{n}}$. Based on the discussion in \cite{Boyd11} on stopping criteria for \admm, the dual stopping criterion for \mialm~is chosen as in \eqref{eq:tol} such that $\mathbf{tol}_d = \epsilon _r \norm{Y_{k+1}}_F$; and the primal stopping criterion for \mialm~is chosen as $\norm{\proj{D}-(L_{k+1}+S_{k+1})}_F \leq \mathbf{tol}_p$, where $\mathbf{tol}_p = \epsilon _r \max\{\norm{L_{k+1}}_F,\norm{S_{k+1}}_F, \norm{\proj{D}}_F\}$, and $\epsilon_r = 5 \times 10^{-4}$. The stopping criterion for \rpcab~is set as $|\Tr(L^*_\rho)-n|/n\leq 0.01$.

The following two cases are considered when we compare the low-rank component output by \admip~with those generated by \mialm$(1/\sqrt{n})$ and by \rpcab. For each $n\in\{100, 200, 300, 400, 500\}$, random networks are generated as described in Section~\ref{sec:generation} for $\alpha\in\{0.6,0.7\ldots,1\}$. 
For each $(n,\alpha)$ setting, we generated 10 random graphs, and corresponding $D$. In \textbf{Case 1}, all the entries of $D$ are observed, i.e., $p_0=1$, and in \textbf{Case 2}, we assume that $\cE$ is partially observable; hence, there are unobserved entries in $D$, i.e., $p_0<1$.

Let $\bar{L}$ represent the underlying clustering in $\cG$, i.e., $\bar{L}=D-\bar{S}$ for $\bar{S}$ defined in \eqref{eq:Sbar}, and $L^*$ is the optimal low-rank component computed by one of the algorithms mentioned above. By definition, $\bar{L}$ is BDO with $r$ diagonal blocks, each of size $n_{\ell}\times n_\ell$ for $\ell=1,...,r$. For $\ell_1,\ell_2\in\{1,...,r\}$, define matrices $\bar{B}_{\ell_1,\ell_2} = \left(\bar{L}_{ij}\right)_{i\in\cN_{\ell_1},j\in\cN_{\ell_2}}\in\reals^{n_{\ell_1}\times n_{\ell_2}}$, and $B^*_{\ell_1,\ell_2} = \left(L^*_{ij}\right)_{i\in\cN_{\ell_1},j\in\cN_{\ell_2}}\in\reals^{n_{\ell_1}\times n_{\ell_2}}$.
Clearly, $\bar{\cB}_{\ell,\ell} = \mathbf{E}_{n_\ell}$ for $\ell\in\{1,\ldots,r\}$, and $\bar{B}_{\ell_1,\ell_2} = \mathbf{0}_{n_{\ell_1}\times n_{\ell_2}}$ if $\ell_1\neq\ell_2$, where $\mathbf{E}_n\in\reals^{n \times n}$ is a matrix of ones. For all $1\leq \ell_1,\ell_2\leq r$, define $R_{\ell_1,\ell_2}:=\norm{\bar{B}_{\ell_1,\ell_2}-B_{\ell_1,\ell_2}^*}_F$. Given a random graph corresponding to $(n,\alpha)$, for each algorithm \admip, \mialm$(1/\sqrt{n})$, and \rpcab, we compute five different statistics. The first three statistics are the maximum, minimum and average of $\left\{\frac{R_{\ell,\ell}}{n_\ell}\right\}_{\ell\in\{1,\ldots,r\}}$, and are denoted by $s_{\max}$, $s_{\min}$, and $s_{\mathrm{av}}$, respectively. The fourth statistic is $s_{\mathrm{off}}:=\frac{\sqrt{\sum_{(\ell_1,\ell_2):\ell_1 < \ell_2}R_{\ell_1,\ell_2}^2}}{\sqrt{\sum_{(\ell_1,\ell_2):\ell_1 < \ell_2}n_{\ell_1} n_{\ell_2}}}$. These first four statistics show how close $L^*$ to the true clustering encoded by the BDO matrix $\bar{L}$. The fifth statistic $s_\mathrm{f}$ is about the fraction of clusters recovered correctly. For $\ell\in\{1,\ldots,r\}$, define
\begin{equation}
\label{eq: FN}
E_{\ell} := \frac{R_{\ell,\ell}}{n_{\ell}},\quad E_{\ell}^c := \frac{\sqrt{\sum_{t: t\neq\ell}R_{\ell,t}^2}}{\sqrt{\sum_{t:t\neq \ell}n_\ell n_t}}.
\end{equation}
We call cluster $\ell$ ``recovered" if $E_{\ell} < \tau_1$ and $E_{\ell}^c < \tau_2$. We set $\tau_1 = 0.4$ and $\tau_2 = 0.1$ for all three algorithms. Let $\bar{r}$ denote the number of recovered clusters, i.e., $\bar{r}:=|\{\ell:\ E_{\ell} < \tau_1,\ E_{\ell}^c < \tau_2\}|$. The fifth statistic reported is $s_\mathrm{f}:=\frac{\bar{r}}{r}$. Note that all five statistics take values in $[0,1]$ interval.

Given $(n,\alpha)$ and $p_0\in\{1,0.9,0.8\}$, the underlying clustering of each 10 random graphs are estimated using \admip, \mialm, and \rpcab. Table~\ref{tab:exp1_p1}, Table~\ref{tab:exp1_p09}, and Table~\ref{tab:exp1_p08} report the averages of 5 statistics: $s_{\max},s_{\min},s_{\mathrm{av}},s_{\mathrm{off}},s_\mathrm{f}$ over the 10 instances for $p_0=1$, $p_0=0.9$, and $p_0=0.8$, respectively. Numerical results show that increasing $n$, and/or decreasing $\alpha$ 
adversely affect the performances of all three methods. However, the negative impact is more serious for \mialm~and \rpcab. Indeed, the results corresponding to $s_\mathrm{f}$ statistic show that while \admip~can detect more than $\%90$ of clusters all the time, $s_\mathrm{f}$ values for \mialm, and \rpcab~decrease significantly (there are many instances for which $s_\mathrm{f}$ is $0$) when $n$ and $\alpha$ change as discussed above. By investigating the results carefully, we see that usually the large values of $E_{\ell}$ cause the failure of \admip~and \mialm. But for \rpcab, the failure is mainly due to $E_{\ell}^c > \tau_2$. Although $E_{\ell}^c$ are not reported, one can drive this result by comparing $s_{\mathrm{off}}$ values corresponding to different scenarios. Let $\rho_0=1/\sqrt{n}$. Indeed, all most all the time when $\alpha<0.9$ and/or $n\geq 300$, low-rank component of \mialm$(\rho_0)$ solution, $L^*_{\rho_0}$, violates $\Tr(L^*_{\rho_0})=n$ condition, which causes $E_\ell\approx 1$, i.e., $(L^*_{\rho_0})_{ij}\approx 0$ for all $i,j\in\cN_\ell$, for all $\ell$ such that $n_\ell$ is \emph{small}; hence, \mialm$(\rho_0)$ \emph{cannot detect small size clusters}. To overcome this issue, Chen et al.~\cite{Jalali14} proposed bisection on $\rho$. Although this approach reduces the error $E_\ell$ significantly, it causes some entries $(L^*_\rho)_{ij}\approx 1$ such that $i\in\cN_{\ell_1}$, $j\in\cN_{\ell_2}$ and $\ell_1\neq\ell_2$; hence, \emph{merging two different clusters}. Intuitively, the reason why our model in \eqref{eq:new_tr} works better is that it considers both type of errors at the same time in a more tighter formulation than RPCA in \eqref{eq:rpca_p}.

Next, we compared cpu times required for \admip~and \rpcab~to terminate on randomly generated networks as in Section~\ref{sec:generation} with $n\in\{500, 1000\}$, $\alpha=0.95$ and $p_0=0.9$. For each $n$, we generated 5 instances; Table~\ref{tab:cpu} displays the averages of $\mathbf{cpu}$, $\mathbf{svd}$ and $\mathbf{iter_B}$ statistics over 5 instances, where $\mathbf{cpu}$, $\mathbf{svd}$ and $\mathbf{iter_B}$ denote runtime (in seconds), total number of SVD computations, and the number of \mialm~calls within \rpcab, respectively.
\subsubsection{ADMIPC vs Modularity Maximization}
\label{sec:modularity_max}
In this section, we compare \admip~with \texttt{Louvain} method~\cite{Louvain08_stat}. Let $\tau_d = 0.05$ and $L^*$ be the optimal low-rank component computed by \admip. If $ \abs{L^* _{ii}-1} > \tau_d$ for some $i=1,...,n$, we declare failure; otherwise, $T^*\in\mathbb{S}_n$ is constructed as follows: $T^*_{ij}=1$ if $L^*_{ij}\geq \bar{\tau}$, and $T^*_{ij}=0$ otherwise, where $\bar{\tau} = 0.55$. Based on $T^*$, we put nodes $i$ and $j$ in the same cluster if $T^*_{ij} = 1$. 
We compare the clusterings generated by \admip, and \texttt{Louvain} method with the ground truth using three different measures of similarity:  
Jaccard index, normalized mutual information~($\mathrm{NMI_{SG}}$) using 
Strehl and Ghosh normalization~\cite{Strehl03}, and portion of exactly recovered clusters~(PERC). All three measures take values in $[0,1]$ interval, and values close to 1 correspond to desirable clusterings. Let $\cG=(\cN,\cE)$ denote the network, $\cC=\{\cN_i\}_{i=1}^r$, which is a partition of $\cN$, represent the ground truth, and $\cC'=\{\cN'_j\}_{j=1}^{r'}$ represent the clustering computed by an algorithm.
\begin{enumerate}
\item \textbf{Jaccard index}: Let $a$ be the number of node pairs that belong to the same clusters in both $\cC$ and $\cC^{\prime}$, $b$ be the number of pairs that are in the same cluster in $\cC$ but in different clusters in $\cC^{\prime}$, and $c$ be the number of pairs that are in the same cluster in $\cC'$ but in different clusters in $\cC$. The Jaccard's index is defined as $\frac{a}{a+b+c}$. It 
has many applications in geology and ecology~\cite{Wagner07}; but it is 
a sensitive measure~\cite{Mill85}. 
The Jaccard's index is in $[0,1]$ interval. It is 1 when $\cC$ and $\cC^{\prime}$ are exactly the same, and equal to 0 when there is no common pair classified in the same cluster in both $\cC$ and $\cC^{\prime}$.
\item \textbf{Normalized Mutual Information (NMI)}: This measure is  based on information theory, and it quantifies the reduction in our uncertainty about one cluster if we know the other one~\cite{Vinh10_ML}. 
    Define $n_i := \abs{\cN_i}$ for $i=1,\ldots,r$, $n'_j := \abs{\cN'_j}$ for $j=1,\ldots,r'$, and $m_{ij} = \abs{\cN_i \cap \cN'_j}$ for all $i,j$. Then the mutual information $\cI(\cC,\cC^{\prime}) := \sum_{i=1}^{r}\sum_{j=1}^{r'} \frac{m_{ij}}{n} \log_2\left(\frac{m_{ij}/n}{n_i n'_j /n^2}\right)$. By normalizing the mutual information, we force it to be between fixed ranges as well as improving its sensitivity~\cite{Vinh10_ML,Wu09}. Let $\cH (\cC) := - \sum_{i=1}^r \frac{n_i}{n}\log_2\left(\frac{n_i}{n}\right)$ represent the entropy associated with clustering $\cC$ and $\cH (\cC^{\prime}) := - \Sigma_{j=1}^{r'}\frac{n'_j}{n}\log_2\left(\frac{n'_j}{n}\right)$ represent the entropy associated with clustering $\cC^{\prime}$. There are different ways of normalizing but we use the method introduced by Strehl and Ghosh~\cite{Strehl03}. In this method, $\mathrm{NMI_{SG}} = \frac{\cI(\cC,\cC^{\prime})}{\sqrt{\cH (\cC)\cH (\cC^{\prime})}} $ is between 0 and 1. 
\item \textbf{Portion of Exactly Recovered Clusters (PERC)}: This measure is a secondary measure and we introduced it to make the comparison in case of tightness in the other measures. Let 
    $\bar{r}$ represent the total number of clusters that are both in $\cC$ and $\cC'$, i.e. the number of clusters identified by the algorithm correctly. Then PERC is equal to $\frac{\bar{r}}{r}$.
\end{enumerate}
For each $n\in\{100, 200, 300, 400, 500\}$ random networks are generated as described in Section~\ref{sec:generation} for $\alpha\in\{0.5,0.6,\ldots,1\}$. 
For each $(n,\alpha)$ setting, we generated 20 random graphs, and 
compute two clusterings using \admip~and \texttt{Louvain} method. When $p_0<1$, if an edge is not observable, then we set the corresponding entry to 0 in the incidence matrix $D$ for \texttt{Louvain} method. For a fixed $(n,\alpha)$ setting and $p_0\in\{1,0.9,0.8\}$, each of the three measures 
are evaluated on the 20 clusterings generated by \admip~on 20 random instances. The mean of these values are reported in Table~\ref{tab:exp2_p1}, Table~\ref{tab:exp2_p09}, and Table~\ref{tab:exp2_p08} for $p_0=1$, $p_0=0.9$, and $p_0=0.8$, respectively. 
Note that the output of \texttt{Louvain} method depends on the initial ordering of the nodes. Hence, for the same graph, this method can generate different clusterings for different orderings. Therefore, for each 20 random graphs corresponding to fixed $(n,\alpha)$, we run \texttt{Louvain} method for 200 different ordering of nodes (generated randomly such that each ordering is equally likely). 
Numerical results show that our method outperforms \texttt{Louvain} method almost every time for all the three measures. It is important to note that increasing $n$ adversely affects the performance of both methods. Moreover, for \texttt{Louvain} method, the clustering quality decreases significantly as $\alpha$ decreases, i.e., the variation among the cluster sizes increases, while it is not the case for our method, and the clustering quality is not impacted significantly.
Analyzing the results we find that for a fixed $n$, \texttt{Louvain} method tends to merge small clusters when $\alpha$ is small. 
On the other hand, \admip~does not show any trend for the first two measures for changing $\alpha$, and almost every time correctly identifies even small size clusters. 
When $\alpha$ is fixed and $n$ increases, the clustering performance of \texttt{Louvain} method decreases again, which agrees with the discussion on resolution limit. As $n$ increases, the number of small size clusters increases in the ground truth as well, and more clusters are merged together by \texttt{Louvain} method, while the number of isolated nodes which originally belong to different clusters increases for \admip. In summary, there are two key points which suggest that \admip~is more reliable than \texttt{Louvain} method. First, \admip~works well even for small values of $\alpha$. Second, by increasing $n$, the performance of both algorithms decrease: \texttt{Louvain} method tends to merge smaller clusters, while our algorithm generates some isolated nodes. But as discussed in~\cite{Mill85}, generating some isolated nodes is less severe than merging some clusters, which makes \admip~more reliable.
\enlargethispage{-19\baselineskip}
\section{Conclusion and Future Work}
We proposed a convex optimization model for detecting nonoverlapping clusters in a partially observed undirected networks; and developed an \admm~algorithm to solve it. Since our formulation is tighter than the robust PCA formulation proposed in~\cite{Jalali14}, we were able to find the true clustering even when the robust PCA formulation failed in our numerical tests. Moreover, our method is not sensitive to moderate changes in variance among cluster sizes, and on the randomly generated networks outperformed \texttt{Louvain} method, which maximizes the modularity and suffers from resolution limit. Extending our formulation to cluster \emph{overlapping} communities in \emph{weighted} networks is a potentially important future research direction. Due to limited space and time, we could not include computational results on real datasets; but they will soon be made available online at \url{http://www2.ie.psu.edu/aybat/codes.html}.
\bibliographystyle{siam}
\bibliography{All}

\begin{table*}[h!]
\begin{center}
{\scriptsize
\begin{tabular}{c|c| lllll | lllll | lllll}
\multicolumn{2}{c}{}
&\multicolumn{5}{c}{\admip}
&\multicolumn{5}{c}{\mialm$(1/\sqrt{n})$}
&\multicolumn{5}{c}{\rpcab}\\
\multicolumn{1}{c|}{$n$}&\multicolumn{1}{c|}{$\alpha$}
& $s_{\max}$ & $s_{\min}$ & $s_{\mathrm{av}}$ & $s_{\mathrm{off}}$ & $s_\mathrm{f}$
& $s_{\max}$ & $s_{\min}$ & $s_{\mathrm{av}}$ & $s_{\mathrm{off}}$ & $s_\mathrm{f}$
& $s_{\max}$ & $s_{\min}$ & $s_{\mathrm{av}}$ & $s_{\mathrm{off}}$ & $s_\mathrm{f}$\\
\hline
	&1	    &0.00&0.00&0.00&0.00&\textbf{1.00}&0.01&0.00&0.00&0.00&\textbf{1.00}&0.01&0.00&0.00&0.00&\textbf{1.00}\\
	&0.9	&0.00&0.00&0.00&0.00&\textbf{1.00}&0.02&0.00&0.00&0.00&\textbf{1.00}&0.02&0.00&0.00&0.00&\textbf{1.00}\\
100	&0.8	&0.00&0.00&0.00&0.00&\textbf{1.00}&0.29&0.00&0.06&0.00&\textbf{0.90}&0.02&0.00&0.00&0.00&\textbf{1.00}\\
	&0.7	&0.02&0.00&0.00&0.00&\textbf{1.00}&1.00&0.00&0.25&0.00&\textbf{0.78}&0.04&0.02&0.02&0.02&\textbf{0.90}\\
	&0.6	&0.03&0.00&0.01&0.00&\textbf{1.00}&1.00&0.00&0.40&0.00&\textbf{0.60}&0.15&0.03&0.06&0.04&\textbf{0.80}\\
	\hline
	&1	    &0.00&0.00&0.00&0.00&\textbf{1.00}&0.10&0.00&0.04&0.00&\textbf{1.00}&0.02&0.00&0.01&0.00&\textbf{1.00}\\
	&0.9	&0.01&0.00&0.00&0.00&\textbf{1.00}&1.00&0.00&0.36&0.00&\textbf{0.61}&0.02&0.00&0.00&0.00&\textbf{1.00}\\
200	&0.8	&0.10&0.00&0.01&0.00&\textbf{1.00}&1.00&0.00&0.48&0.00&\textbf{0.50}&0.29&0.04&0.20&0.20&\textbf{0.00}\\
	&0.7	&0.28&0.00&0.05&0.00&\textbf{0.99}&1.00&0.00&0.58&0.00&\textbf{0.40}&0.35&0.02&0.18&0.20&\textbf{0.00}\\
	&0.6	&0.30&0.00&0.04&0.00&\textbf{0.98}&1.00&0.00&0.61&0.00&\textbf{0.40}&0.44&0.00&0.17&0.20&\textbf{0.00}\\
	\hline
	&1	    &0.01&0.00&0.00&0.00&\textbf{1.00}&0.80&0.38&0.57&0.00&\textbf{0.06}&0.00&0.00&0.00&0.00&\textbf{1.00}\\
	&0.9	&0.10&0.00&0.01&0.00&\textbf{1.00}&1.00&0.00&0.57&0.00&\textbf{0.41}&0.28&0.08&0.20&0.20&\textbf{0.00}\\
300	&0.8	&0.43&0.00&0.06&0.00&\textbf{0.96}&1.00&0.00&0.62&0.00&\textbf{0.38}&0.41&0.01&0.19&0.20&\textbf{0.00}\\
	&0.7	&0.37&0.00&0.05&0.00&\textbf{0.98}&1.00&0.00&0.67&0.00&\textbf{0.33}&0.55&0.00&0.17&0.20&\textbf{0.00}\\
	&0.6	&0.36&0.00&0.04&0.00&\textbf{0.97}&1.00&0.00&0.64&0.00&\textbf{0.36}&0.44&0.01&0.18&0.20&\textbf{0.00}\\
	\hline
	&1	    &0.05&0.00&0.00&0.00&\textbf{1.00}&1.00&0.97&1.00&0.00&\textbf{0.00}&0.02&0.00&0.00&0.00&\textbf{1.00}\\
	&0.9	&0.22&0.00&0.02&0.00&\textbf{1.00}&1.00&0.00&0.65&0.00&\textbf{0.34}&0.31&0.01&0.19&0.20&\textbf{0.00}\\
400	&0.8	&0.47&0.00&0.07&0.00&\textbf{0.95}&1.00&0.00&0.70&0.00&\textbf{0.30}&0.56&0.00&0.19&0.20&\textbf{0.00}\\
	&0.7	&0.49&0.00&0.08&0.00&\textbf{0.94}&1.00&0.00&0.69&0.00&\textbf{0.31}&0.56&0.00&0.22&0.20&\textbf{0.00}\\
	&0.6	&0.35&0.00&0.06&0.00&\textbf{0.98}&1.00&0.00&0.65&0.00&\textbf{0.33}&0.55&0.02&0.21&0.20&\textbf{0.00}\\
	\hline
	&1	    &0.12&0.00&0.01&0.00&\textbf{1.00}&1.00&1.00&1.00&0.00&\textbf{0.00}&0.07&0.00&0.00&0.00&\textbf{1.00}\\
	&0.9	&0.31&0.00&0.04&0.00&\textbf{0.99}&1.00&0.00&0.70&0.00&\textbf{0.31}&0.36&0.01&0.19&0.20&\textbf{0.00}\\
500	&0.8	&0.53&0.00&0.08&0.00&\textbf{0.90}&1.00&0.00&0.72&0.00&\textbf{0.29}&0.57&0.00&0.20&0.20&\textbf{0.00}\\
	&0.7	&0.47&0.00&0.08&0.00&\textbf{0.94}&1.00&0.00&0.65&0.00&\textbf{0.35}&0.54&0.01&0.22&0.20&\textbf{0.00}\\
	&0.6	&0.43&0.00&0.07&0.00&\textbf{0.94}&1.00&0.00&0.59&0.00&\textbf{0.42}&0.53&0.01&0.22&0.20&\textbf{0.00}\\
\end{tabular}
\caption{The mean values for 5 statistics for $p_0 =1$}
\label{tab:exp1_p1}
\vspace{2cm}
\begin{tabular}{c|c| lllll | lllll | lllll}
\multicolumn{2}{c}{}
&\multicolumn{5}{c}{\admip} &\multicolumn{5}{c}{\mialm$(1/\sqrt{n})$} &\multicolumn{5}{c}{\rpcab}\\
\multicolumn{1}{c|}{$n$}&\multicolumn{1}{c|}{$\alpha$}
& $s_{\max}$ & $s_{\min}$ & $s_{\mathrm{av}}$ & $s_{\mathrm{off}}$ & $s_\mathrm{f}$
& $s_{\max}$ & $s_{\min}$ & $s_{\mathrm{av}}$ & $s_{\mathrm{off}}$ & $s_\mathrm{f}$
& $s_{\max}$ & $s_{\min}$ & $s_{\mathrm{av}}$ & $s_{\mathrm{off}}$ & $s_\mathrm{f}$\\
\hline
    &1	    &0.00&0.00&0.00&0.00&\textbf{1.00}&0.04&0.00&0.01&0.00&\textbf{1.00}&0.04&0.00&0.01&0.00&\textbf{1.00}\\
	&0.9	&0.00&0.00&0.00&0.00&\textbf{1.00}&0.11&0.00&0.03&0.00&\textbf{1.00}&0.03&0.00&0.01&0.00&\textbf{1.00}\\
100	&0.8	&0.00&0.00&0.00&0.00&\textbf{1.00}&0.72&0.00&0.16&0.00&\textbf{0.82}&0.00&0.00&0.00&0.00&\textbf{1.00}\\
	&0.7	&0.03&0.00&0.01&0.00&\textbf{1.00}&1.00&0.00&0.34&0.00&\textbf{0.64}&0.07&0.00&0.02&0.00&\textbf{1.00}\\
	&0.6	&0.08&0.00&0.02&0.00&\textbf{1.00}&1.00&0.00&0.40&0.00&\textbf{0.60}&0.21&0.05&0.11&0.10&\textbf{0.50}\\
	\hline
	&1	    &0.00&0.00&0.00&0.00&\textbf{1.00}&0.39&0.05&0.18&0.00&\textbf{0.94}&0.00&0.00&0.00&0.00&\textbf{1.00}\\
	&0.9	&0.01&0.00&0.00&0.00&\textbf{1.00}&1.00&0.00&0.48&0.00&\textbf{0.51}&0.02&0.00&0.00&0.00&\textbf{1.00}\\
200	&0.8	&0.18&0.00&0.03&0.00&\textbf{0.99}&1.00&0.00&0.53&0.00&\textbf{0.45}&0.32&0.12&0.21&0.19&\textbf{0.00}\\
	&0.7	&0.33&0.00&0.06&0.00&\textbf{0.98}&1.00&0.00&0.61&0.00&\textbf{0.40}&0.39&0.06&0.21&0.19&\textbf{0.00}\\
	&0.6	&0.42&0.00&0.06&0.01&\textbf{0.93}&1.00&0.00&0.66&0.00&\textbf{0.34}&0.65&0.02&0.25&0.19&\textbf{0.00}\\
	\hline
	&1	    &0.03&0.00&0.00&0.00&\textbf{1.00}&1.00&0.75&0.93&0.00&\textbf{0.00}&0.02&0.00&0.00&0.00&\textbf{1.00}\\
	&0.9	&0.19&0.00&0.03&0.00&\textbf{1.00}&1.00&0.00&0.65&0.00&\textbf{0.33}&0.32&0.12&0.21&0.19&\textbf{0.00}\\
300	&0.8	&0.40&0.00&0.08&0.00&\textbf{0.97}&1.00&0.00&0.67&0.00&\textbf{0.33}&0.47&0.02&0.22&0.19&\textbf{0.00}\\
	&0.7	&0.44&0.00&0.07&0.01&\textbf{0.95}&1.00&0.00&0.69&0.00&\textbf{0.31}&0.62&0.00&0.22&0.19&\textbf{0.00}\\
	&0.6	&0.32&0.00&0.05&0.00&\textbf{0.99}&1.00&0.00&0.64&0.00&\textbf{0.36}&0.55&0.00&0.22&0.19&\textbf{0.00}\\
	\hline
	&1	    &0.10&0.00&0.01&0.00&\textbf{1.00}&1.00&1.00&1.00&0.00&\textbf{0.00}&0.06&0.00&0.00&0.00&\textbf{1.00}\\
	&0.9	&0.33&0.00&0.06&0.00&\textbf{0.98}&1.00&0.00&0.71&0.00&\textbf{0.30}&0.39&0.10&0.23&0.19&\textbf{0.00}\\
400	&0.8	&0.49&0.00&0.10&0.01&\textbf{0.94}&1.00&0.00&0.71&0.00&\textbf{0.30}&0.49&0.01&0.20&0.20&\textbf{0.00}\\
	&0.7	&0.56&0.00&0.11&0.00&\textbf{0.90}&1.00&0.00&0.69&0.00&\textbf{0.31}&0.60&0.00&0.23&0.19&\textbf{0.00}\\
	&0.6	&0.45&0.00&0.09&0.00&\textbf{0.95}&1.00&0.00&0.67&0.00&\textbf{0.33}&0.55&0.00&0.23&0.19&\textbf{0.00}\\
	\hline
	&1	    &0.19&0.00&0.02&0.00&\textbf{1.00}&1.00&1.00&1.00&0.00&\textbf{0.00}&0.08&0.00&0.01&0.00&\textbf{1.00}\\
	&0.9	&0.39&0.00&0.05&0.00&\textbf{0.98}&1.00&0.00&0.74&0.00&\textbf{0.26}&0.48&0.08&0.24&0.19&\textbf{0.00}\\
500	&0.8	&0.51&0.00&0.10&0.00&\textbf{0.92}&1.00&0.00&0.74&0.00&\textbf{0.25}&0.70&0.00&0.24&0.19&\textbf{0.00}\\
	&0.7	&0.54&0.00&0.10&0.00&\textbf{0.91}&1.00&0.00&0.69&0.00&\textbf{0.31}&0.62&0.01&0.25&0.19&\textbf{0.00}\\
	&0.6	&0.54&0.00&0.09&0.00&\textbf{0.93}&1.00&0.00&0.65&0.00&\textbf{0.33}&0.63&0.00&0.25&0.19&\textbf{0.00}\\
\end{tabular}
\caption{The mean values for 5 statistics for $p_0 =0.9$}
\label{tab:exp1_p09}
}
\end{center}
\vspace{-0.5cm}
\end{table*}

\begin{table*}[th]
\begin{center}
{\scriptsize
\begin{tabular}{c|c| lllll | lllll | lllll}
\multicolumn{2}{c}{}
&\multicolumn{5}{c}{\admip}
&\multicolumn{5}{c}{\mialm$(1/\sqrt{n})$}
&\multicolumn{5}{c}{\rpcab}\\
\multicolumn{1}{c|}{$n$}&\multicolumn{1}{c|}{$\alpha$}
& $s_{\max}$ & $s_{\min}$ & $s_{\mathrm{av}}$ & $s_{\mathrm{off}}$ & $s_\mathrm{f}$ & $s_{\max}$ & $s_{\min}$ & $s_{\mathrm{av}}$ & $s_{\mathrm{off}}$ & $s_\mathrm{f}$ & $s_{\max}$ & $s_{\min}$ & $s_{\mathrm{av}}$ & $s_{\mathrm{off}}$ & $s_\mathrm{f}$\\
\hline
&1	    &0.00&0.00&0.00&0.00&\textbf{1.00}&0.10&0.00&0.04&0.00&\textbf{1.00}&0.01&0.00&0.00&0.00&\textbf{1.00}\\
	&0.9	&0.00&0.00&0.00&0.00&\textbf{1.00}&0.29&0.00&0.08&0.00&\textbf{0.96}&0.00&0.00&0.00&0.00&\textbf{1.00}\\
100	&0.8	&0.02&0.00&0.00&0.00&\textbf{1.00}&1.00&0.00&0.28&0.00&\textbf{0.72}&0.03&0.00&0.01&0.00&\textbf{1.00}\\
	&0.7	&0.02&0.00&0.00&0.00&\textbf{1.00}&1.00&0.00&0.40&0.00&\textbf{0.60}&0.09&0.03&0.05&0.04&\textbf{0.80}\\
	&0.6	&0.15&0.00&0.04&0.00&\textbf{1.00}&1.00&0.00&0.43&0.00&\textbf{0.60}&0.29&0.11&0.18&0.15&\textbf{0.20}\\
	\hline
	&1	    &0.01&0.00&0.00&0.00&\textbf{1.00}&0.88&0.35&0.64&0.00&\textbf{0.07}&0.00&0.00&0.00&0.00&\textbf{1.00}\\
	&0.9	&0.11&0.00&0.02&0.00&\textbf{1.00}&1.00&0.00&0.58&0.00&\textbf{0.42}&0.16&0.03&0.06&0.03&\textbf{0.80}\\
200	&0.8	&0.30&0.00&0.06&0.00&\textbf{0.99}&1.00&0.00&0.61&0.00&\textbf{0.40}&0.37&0.14&0.23&0.18&\textbf{0.00}\\
	&0.7	&0.37&0.00&0.09&0.01&\textbf{0.96}&1.00&0.00&0.63&0.00&\textbf{0.37}&0.46&0.09&0.25&0.18&\textbf{0.00}\\
	&0.6	&0.50&0.00&0.09&0.01&\textbf{0.92}&1.00&0.00&0.70&0.00&\textbf{0.30}&0.61&0.02&0.24&0.18&\textbf{0.00}\\
	\hline
	&1	    &0.10&0.00&0.02&0.00&\textbf{1.00}&1.00&0.99&1.00&0.00&\textbf{0.00}&0.05&0.00&0.01&0.00&\textbf{1.00}\\
	&0.9	&0.26&0.00&0.06&0.00&\textbf{1.00}&1.00&0.00&0.72&0.00&\textbf{0.29}&0.34&0.14&0.22&0.17&\textbf{0.00}\\
300	&0.8	&0.62&0.00&0.15&0.01&\textbf{0.85}&1.00&0.00&0.70&0.00&\textbf{0.30}&0.64&0.12&0.28&0.18&\textbf{0.00}\\
	&0.7	&0.57&0.00&0.10&0.01&\textbf{0.89}&1.00&0.00&0.73&0.00&\textbf{0.27}&0.61&0.00&0.24&0.19&\textbf{0.00}\\
	&0.6	&0.42&0.00&0.08&0.01&\textbf{0.94}&1.00&0.00&0.65&0.00&\textbf{0.36}&0.65&0.00&0.26&0.17&\textbf{0.00}\\
	\hline
	&1	    &0.17&0.00&0.04&0.00&\textbf{1.00}&1.00&1.00&1.00&0.00&\textbf{0.00}&0.14&0.03&0.06&0.03&\textbf{0.80}\\
	&0.9	&0.36&0.00&0.06&0.00&\textbf{0.99}&1.00&0.00&0.77&0.00&\textbf{0.22}&0.45&0.14&0.25&0.17&\textbf{0.00}\\
400	&0.8	&0.63&0.00&0.14&0.01&\textbf{0.87}&1.00&0.00&0.74&0.00&\textbf{0.26}&0.67&0.01&0.26&0.20&\textbf{0.00}\\
	&0.7	&0.58&0.00&0.11&0.01&\textbf{0.90}&1.00&0.00&0.70&0.00&\textbf{0.30}&0.69&0.00&0.27&0.18&\textbf{0.00}\\
	&0.6	&0.57&0.00&0.12&0.01&\textbf{0.89}&1.00&0.00&0.67&0.00&\textbf{0.33}&0.71&0.06&0.29&0.18&\textbf{0.00}\\
	\hline
	&1	    &0.28&0.00&0.04&0.00&\textbf{1.00}&1.00&1.00&1.00&0.00&\textbf{0.00}&0.29&0.14&0.22&0.16&\textbf{0.00}\\
	&0.9	&0.43&0.00&0.07&0.00&\textbf{0.98}&1.00&0.00&0.79&0.00&\textbf{0.20}&0.57&0.13&0.27&0.18&\textbf{0.00}\\
500	&0.8	&0.66&0.00&0.13&0.01&\textbf{0.91}&1.00&0.00&0.76&0.00&\textbf{0.25}&0.66&0.00&0.22&0.21&\textbf{0.00}\\
	&0.7	&0.54&0.00&0.11&0.00&\textbf{0.92}&1.00&0.00&0.69&0.00&\textbf{0.31}&0.63&0.02&0.25&0.18&\textbf{0.00}\\
	&0.6	&0.53&0.00&0.11&0.00&\textbf{0.91}&1.00&0.00&0.67&0.00&\textbf{0.33}&0.64&0.06&0.28&0.17&\textbf{0.00}\\	
\end{tabular}
}
\end{center}
\vspace{-0.5cm}
\caption{The mean values for 5 statistics for $p_0 =0.8$}
\label{tab:exp1_p08}
\vspace{-0.45cm}
\end{table*}
\begin{table*}[th]
\begin{center}
{\scriptsize
\begin{tabular}{c|c| lllllll | llllllll }
\multicolumn{2}{c}{}
&\multicolumn{7}{c}{\admip}&\multicolumn{8}{c}{\rpcab}\\
\multicolumn{1}{c|}{$n$}&\multicolumn{1}{c|}{instance \#}
& $s_{\max}$ & $s_{\min}$ & $s_{\mathrm{av}}$ & $s_{\mathrm{off}}$ & $s_\mathrm{f}$ &$\mathrm{cpu}$&$\mathrm{svd}$& $s_{\max}$ & $s_{\min}$ & $s_{\mathrm{av}}$ & $s_{\mathrm{off}}$ & $s_\mathrm{f}$&$\mathrm{cpu}$&$\mathrm{svd}$& $\mathrm{iter_B}$
\\
\hline
&1&0.21&0.00&0.04&0.00&1.00&\textbf{3.8}&\textbf{30}&0.30&0.15&0.22&0.18&0.00&\textbf{13.6}&\textbf{61}&3\\
&2&0.17&0.00&0.03&0.00&1.00&\textbf{3.8}&\textbf{31}&0.27&0.12&0.21&0.18&0.00&\textbf{13.3}&\textbf{58}&3\\
500
&3&0.16&0.00&0.02&0.00&1.00&\textbf{3.9}&\textbf{33}&0.27&0.10&0.20&0.18&0.00&\textbf{13.4}&\textbf{58}&3\\
&4&0.20&0.00&0.02&0.00&1.00&\textbf{3.4}&\textbf{29}&0.31&0.12&0.21&0.18&0.00&\textbf{13.3}&\textbf{60}&3\\
&5&0.19&0.00&0.03&0.00&1.00&\textbf{3.5}&\textbf{31}&0.33&0.14&0.22&0.18&0.00&\textbf{13.8}&\textbf{60}&3\\
\hline
&1&0.57&0.00&0.05&0.01&0.98&\textbf{26.7}&\textbf{40}&0.55&0.00&0.27&0.22&0.00&\textbf{159.3}&\textbf{92}&4\\
&2&0.37&0.00&0.03&0.00&1.00&\textbf{25.7}&\textbf{37}&0.50&0.00&0.26&0.22&0.00&\textbf{163.2}&\textbf{92}&4\\
1000
&3&0.31&0.00&0.04&0.00&1.00&\textbf{25.9}&\textbf{39}&0.59&0.00&0.27&0.22&0.00&\textbf{162.2}&\textbf{92}&4\\
&4&0.34&0.00&0.04&0.01&1.00&\textbf{27.2}&\textbf{41}&0.45&0.00&0.26&0.22&0.00&\textbf{161.5}&\textbf{90}&4\\
&5&0.52&0.00&0.04&0.00&0.98&\textbf{26.0}&\textbf{38}&0.58&0.00&0.26&0.22&0.00&\textbf{162.9}&\textbf{93}&4\\
\end{tabular}%
}
\end{center}
\vspace{-0.5cm}
\caption{The cpu times, total svd numbers, and 5 statistics for $p_0 =0.9$ and $\alpha=0.95$}
\label{tab:cpu}
\vspace{-0.1cm}
\end{table*}
\begin{table*}[htb]
\begin{minipage}{\textwidth}
\begin{center}
{\scriptsize
\begin{tabular}{cc| lllll | lllll}
\multicolumn{2}{c}{}
&\multicolumn{5}{c}{\texttt{Louvain}}&\multicolumn{5}{c}{\admip}\\
&\multicolumn{1}{c|}{\diagbox[width=3em]{~$\alpha$}{$n$}}
&\multicolumn{1}{c}{100} &\multicolumn{1}{c}{200} &\multicolumn{1}{c}{300} &\multicolumn{1}{c}{400} &\multicolumn{1}{c|}{500} &\multicolumn{1}{c}{100} &\multicolumn{1}{c}{200} &\multicolumn{1}{c}{300} &\multicolumn{1}{c}{400} &\multicolumn{1}{c}{500}\\
\hline
&1
&99.9    &98.9    &95.5    &90.0    &84.1    &100	  &100	  &100	   &100  	 &99.9\\
&0.9
&99.8    &93.9    &78.2    &67.4    &58.8    &100	  &100	  &100	  &99.9	 &99.8\\
$\mathrm{Jaccard}$
&0.8
&99.5    &83.9    &73.4    &67.9    &66.4    &100	 &99.9	 &99.9	  &99.9	 &99.9\\
$\mathrm{Index}$
&0.7
&98.2    &82.0    &78.5   &76.7    &76.8    &100	  &99.9	  &99.9	  &99.9	 &99.9\\
&0.6
&94.3    &86.1    &86.1    &86.9    &86.8    &99.9	  &99.9	  &99.9	 &99.9	 &99.9\\
&0.5
&94.3    &88.7    &89.9    &88.9    &87.9    &100	  &99.9	  &99.9	&99.9	&94.9\\
\hline
&1
&100    &99.8    &99.3    &98.6    &97.9   &100	  &100	  &100	   &100 	 &99.9\\
&0.9
&99.9    &98.6    &94.5    &91.2    &88.1    &100	  &100	  &100	  &99.9	 &99.8\\
$\mathrm{NMI_{SG}}$
&0.8
&99.8    &93.7    &88.9    &86.5    &85.5    &100	 &99.9	  &99.9	  &99.7	 &99.7\\
&0.7
&99.0    &89.4    &87.2    &86.2    &86.2    &100	  &99.8	  &99.7	  &99.7	 &99.7\\
&0.6
&95.6    &88.0    &87.9    &88.0    &87.8    &99.9	  &99.8	  &99.8	  &99.8	 &99.8\\
&0.5
&93.1    &87.2    &88.0    &86.8    &86.2    &100 	&99.8	 &99.8	  &99.9	    &99.9\\
\hline
&1
&99.9    &98.8    &95.0    &89.1    &82.6    &100	  &100	  &100	   &100 	&99.8\\
&0.9
&99.7    &90.1    &51.8    &32.1    &18.1    &100	  &100	  &100	   &98.5	&94.8\\
$\mathrm{PERC}$
&0.8
&99.2    &49.7    &20.3    &9.95    &6.30    &100	  &99.5	  &98.6	    &86.0	 &88.4\\
&0.7
&92.9    &27.4    &9.50    &8.37    &8.38    &100	  &96.5	   &88.6  	&86.4	 &85.0\\
&0.6
&65.3    &15.0    &10.5    &10.7    &10.3    &99   	  &94	   &90.8 	&91.6	 &91.1\\
&0.5
&42.5    &15.3    &13.1    &12.3    &13.4    &100	  &93.7	   &93.3 	&93.3	 &86.1\\
\end{tabular}
\caption{The mean values for 3 measures in $\%$ when $p_0=1$}
\label{tab:exp2_p1}
\vspace{1cm}
\begin{tabular}{cc| lllll | lllll}
\multicolumn{2}{c}{}
&\multicolumn{5}{c}{\texttt{Louvain}}&\multicolumn{5}{c}{\admip}\\
&\multicolumn{1}{c|}{\diagbox[width=3em]{~$\alpha$}{$n$}}
&\multicolumn{1}{c}{100} &\multicolumn{1}{c}{200} &\multicolumn{1}{c}{300} &\multicolumn{1}{c}{400} &\multicolumn{1}{c|}{500} &\multicolumn{1}{c}{100} &\multicolumn{1}{c}{200} &\multicolumn{1}{c}{300} &\multicolumn{1}{c}{400} &\multicolumn{1}{c}{500}\\
\hline
&1
&99.9    &99.1    &96.5    &92.8    &85.0    &100	  &100	  &100	  &100  &99.9\\
&0.9
&99.9    &95.0    &78.4    &67.4    &59.7    &100	  &100 	   &99.9 	&99.9	&99.7\\
$\mathrm{Jaccard}$
&0.8
&99.7    &85.0    &73.6    &68.7    &67.5    &100	  &99.9	  &99.9	  &99.9	 &99.9\\
$\mathrm{Index}$
&0.7
&98.4    &83.4    &78.1    &77.8    &77.3    &99.9	  &99.9	   &99.9     &99.9	&99.9\\
&0.6
&93.6    &85.9    &86.8    &87.0    &86.9    &99.9	  &99.9	   &99.9 	&99.9	 &99.9\\
&0.5
&94.6    &90.2    &90.6    &88.7    &88.1    &100	  &99.9	   &99.9 	&99.9	 &99.9\\
\hline
&1
&99.9    &99.8    &99.5    &99.0    &98.1    &100	  &100	  &100	   &100  	 &99.9\\
&0.9
&99.9    &98.8    &94.4    &91.1    &88.1    &100	  &100	  &99.9	  &99.9	&99.6\\
$\mathrm{NMI_{SG}}$
&0.8
&99.9    &93.9    &88.6    &86.5    &85.8    &100	  &99.9	  &99.8	  &99.6	 &99.6\\
&0.7
&99.1    &89.9    &87.0    &86.4    &86.3    &99.9	 &99.7	   &99.6 	&99.7	 &99.7\\
&0.6
&95.2    &87.7    &88.1    &87.9    &87.9    &99.9	 &99.6	   &99.7 	&99.8	 &99.8\\
&0.5
&93.3    &88.2    &88.5    &86.6    &86.3    &100	 &99.7	   &99.8 	&99.8	 &99.9\\
\hline
&1
&99.9    &99.0    &96.2    &92.2    &83.3    &100	  &100	  &100	   &100  	 &99.4\\
&0.9
&99.9    &91.1    &51.8    &30.9    &17.2    &100	  &100	  &99.3	  &97.5	 &90.2\\
$\mathrm{PERC}$
&0.8
&99.5    &50.1    &17.1    &10.3    &5.73    &100	  &98	   &93.3  	&81.8	 &84.2\\
&0.7
&93.3    &26.0    &7.97    &7.40    &7.51    &99.0 	  &93.0	   &87.0	  &84.1	 &82.9\\
&0.6
&63.5    &13.0    &9.83    &9.31    &10.0    &99	      &90.5	   &87.5  	&91.2	 &87.6\\
&0.5
&43.2    &16.6    &13.8    &11.5    &13.1    &100	  &91.2	   &90.5 	&92.7	 &90\\
\end{tabular}
\caption{The mean values for 3 measures in $\%$ when $p_0=0.9$}
\label{tab:exp2_p09}
\vspace{1cm}
\begin{tabular}{cc| lllll | lllll}
\multicolumn{2}{c}{}
&\multicolumn{5}{c}{\texttt{Louvain}}&\multicolumn{5}{c}{\admip}\\
&\multicolumn{1}{c|}{\diagbox[width=3em]{~$\alpha$}{$n$}}
&\multicolumn{1}{c}{100} &\multicolumn{1}{c}{200} &\multicolumn{1}{c}{300} &\multicolumn{1}{c}{400} &\multicolumn{1}{c|}{500} &\multicolumn{1}{c}{100} &\multicolumn{1}{c}{200} &\multicolumn{1}{c}{300} &\multicolumn{1}{c}{400} &\multicolumn{1}{c}{500}\\
\hline
&1
&99.9    &99.4    &97.4    &93.7    &86.5    &100	  &100	  &100	   &100 	 &99.8\\
&0.9
&99.9    &95.9    &78.9    &66.4    &60.2    &100	  &100	   &99.9 	&99.7	&99.4\\
$\mathrm{Jaccard}$&0.8
&99.8    &85.1    &73.20    &67.5    &67.4    &100	  &99.8	    &99.8	&99.8	 &99.8\\
$\mathrm{Index}$
&0.7
&97.9    &82.3    &79.3    &75.7    &75.4    &99.9	  &99.8	   &99.9 	&99.9	 &99.9\\
&0.6
&93.4    &87.1    &87.5    &87.6    &87.1    &99.9	  &99.9	   &99.9 	&99.9	 &99.9\\
&0.5
&94.3    &88.5    &89.2    &88.8    &88.5    &99.9	  &99.9	   &99.9 	&99.9	 &99.9\\
\hline
&1
&100    &99.9    &99.6    &99.1    &98.3   &100	     &100	  &100	  &100 	&99.9\\
&0.9
&99.9    &99.0    &94.6    &90.7    &87.9    &100	  &100	   &99.9 	&99.7	 &99.2\\
$\mathrm{NMI_{SG}}$
&0.8
&99.9    &93.7    &88.4    &86.1    &85.3    &100	  &99.8	   &99.6 	&99.4	 &99.4\\
&0.7
&98.8    &89.3    &87.1    &85.7    &85.6    &99.9	  &99.7	   &99.5 	&99.5	 &99.6\\
&0.6
&95.1    &88.3    &88.2    &88.1    &87.8    &99.8	  &99.6	   &99.5 	&99.7	&99.7\\
&0.5
&93.1    &87.2    &87.5    &86.5    &86.3    &99.8	 &99.7	   &99.8 	&99.7	 &99.8\\
\hline
&1
&99.9    &99.3    &97.2   &93.2    &85.1    &100	  &100	  &100	&100  	 &98.8\\
&0.9
&99.9    &91.8    &52.1    &29.9    &14.42   &100	  &100	   &99	  &92.7	 &80.6\\
$\mathrm{PERC}$
&0.8
&99.8    &45.4    &16.9    &8.88    &5.26    &100	  &96.5	   &89.3 	&78.1	 &78.6\\
&0.7
&91.4    &24.6    &8.17    &6.61    &6.73    &99	      &93	   &84	  &77.6	 &81.1\\
&0.6
&64.8    &14.7    &9.61    &8.58    &10.08   &98  	  &89	   &80.4 	&86.2	 &84.6\\
&0.5
&42.7    &14.3    &12.0    &10.7    &12.1    &97  	  &89.3	   &88.8 	&87.7	 &90\\
\end{tabular}
\caption{The mean values for 3 measures in $\%$ when $p_0=0.8$}
\label{tab:exp2_p08}
}
\end{center}
\end{minipage}
\end{table*}
\end{document}